\documentclass[12 pt]{article}
\usepackage[english]{babel}
\usepackage[utf8]{inputenc}
\usepackage{amsmath, amsfonts, amssymb, amsthm}
\usepackage{mathtools}
\usepackage{graphicx}
\usepackage{enumitem}
\usepackage{logicproof}
\usepackage{bussproofs}
\usepackage[dvipsnames]{xcolor}
\usepackage{xfrac}
\usepackage{paracol}
\usepackage{stackengine}
\usepackage{etoolbox}
\usepackage{needspace}
\usepackage{tikz}
\usepackage[style=authoryear, backend=bibtex, maxbibnames=99]{biblatex}
\usepackage{hyperref}
\hypersetup{
    colorlinks=true,
    linkcolor=blue,
    filecolor=blue,
    citecolor=blue,
    urlcolor=blue}

\newtheoremstyle{break}
  {10pt}
  {15pt}
  {}
  {}
  {\bfseries}
  {}
  {\newline}
  {}
\theoremstyle{definition}
\newtheorem{defn}{Definition}[section]
\newtheorem{theorem}{Theorem}[section]
\newtheorem{corollary}{Corollary}[section]
\newtheorem{lemma}[theorem]{Lemma}
\newtheorem{fact}{Fact}[section]

\newtheorem{proposition}[theorem]{Proposition}

\AtBeginEnvironment{defn}{\Needspace{5\baselineskip}}

\newcommand{\K}{\mathbb{K}_3^{+}}

\newcommand{\LP}{\mathbb{LP}^{+}}

\newcommand{\ST}{\mathbb{ST}^{+}}

\newcommand{\TS}{\mathbb{TS}^{+}}

\makeatletter
\newtheorem*{rep@theorem}{\rep@title}
\newcommand{\newreptheorem}[2]{%
\newenvironment{rep#1}[1]{%
 \def\rep@title{#2 \ref{##1}}%
 \begin{rep@theorem}}%
 {\end{rep@theorem}}}
\makeatother

\newcommand{\nmodels}{\Relbar\joinrel\mathrel{|}}

\newcommand{\ms}{\scriptscriptstyle}

\newcommand\Mysim{\mathord{\sim}}

\newcommand\val{\mathord{v^{\star}}}

\begin{document}

\title{$\mathsf{ST}$ and $\mathsf{TS}$ as Product and Sum
}

\author{Quentin Blomet \and Paul \'Egré}
\date{}
 
\maketitle

\begin{abstract}
\noindent The set of $\mathsf{ST}$-valid inferences is neither the intersection, nor the union of the sets of $\mathsf{K}_3$- and $\mathsf{LP}$-valid inferences, but despite the proximity to both systems, an extensional characterization of $\mathsf{ST}$ in terms of a natural set-theoretic operation on the sets of $\mathsf{K}_3$- and $\mathsf{LP}$-valid inferences is still wanting.
In this paper, we show that it is their {\it relational product}. Similarly, we prove that the set of $\mathsf{TS}$-valid inferences can be identified using a dual notion, namely as the {\it relational sum} of the sets of $\mathsf{LP}$- and $\mathsf{K}_3$-valid inferences. We discuss links between these results and the interpolation property of classical logic. We also use those results to revisit the duality between $\mathsf{ST}$ and $\mathsf{TS}$. We present a combined notion of duality on which $\mathsf{ST}$ and $\mathsf{TS}$ are dual in exactly the same sense in which $\mathsf{LP}$ and $\mathsf{K}_3$ are dual to each other.
\end{abstract}

\sloppy

\section{Introduction}
This paper deals with the substructural logics $\mathsf{ST}$ and $\mathsf{TS}$, first introduced under those names in \cite{Cobrerosetal2012}, and also known as $p$-consequence and $q$-consequence logics (\cite{Frankowski2004, Malinowski1990}). These logics, which admit a canonical three-valued characterization (\cite{ripley2012conservatively,cobreros2015vagueness, chemla2019suszko}), are closely related to the well-known systems $\mathsf{LP}$ and $\mathsf{K}_3$, namely Priest's Logic of Paradox \parencite{Asenjo1966, Priest1979} and Kleene's Strong three-valued logic \parencite{Kleene1952-KLEITM}. However, whereas $\mathsf{LP}$-validity and $\mathsf{K}_3$-validity are defined in terms of the preservation of a designated set of truth-values (truth for $\mathsf{K}_3$, non-falsity for $\mathsf{LP}$), $\mathsf{ST}$ and $\mathsf{TS}$ are systems of ``mixed consequence'', in which validity means that the strict truth of the premises (or $\mathsf{K}_3$-truth) entails the non-falsity of the conclusions (or $\mathsf{LP}$-truth), and conversely for $\mathsf{TS}$.

The question this paper proposes to answer is the following: is there a natural characterization of $\mathsf{ST}$-validity and $\mathsf{TS}$-validity, directly in terms of the valid inferences of $\mathsf{LP}$ and the valid inferences of $\mathsf{K}_3$? Some results exist that relate the \textit{meta}inferences of $\mathsf{ST}$ to the inferences of $\mathsf{LP}$ (\cite{Barrioetal2015, DicherPaoli2019}), or indeed the metainferences of $\mathsf{TS}$ to the inferences of $\mathsf{K}_3$ (\cite{Cobrerosetal2020}). Other results show that classical logic (and therefore $\mathsf{ST}$) is the least extension of Dunn-Belnap's four-valued logic ($\mathsf{FDE}$) that contains both $\mathsf{K}_3$ and $\mathsf{LP}$ \parencite{Albuquerqueetal2017, prenosil_2023}. However, in this paper we are interested in whether $\mathsf{ST}$-valid and $\mathsf{TS}$-valid inferences can be defined in terms of some more direct set-theoretic operation on the sets of valid inferences of $\mathsf{K}_3$ and $\mathsf{LP}$. 

It is known that the set of $\mathsf{ST}$-valid inferences is neither the intersection, nor the union of the sets of $\mathsf{K}_3$- and $\mathsf{LP}$-valid inferences (viz. \cite{Chemlaetal2017}), but which operation relates them exactly remained to be determined. The first result we establish in this paper --- originally a conjecture formulated by the second author in 2011, but long left sitting in a drawer --- is that $\mathsf{ST}$ is the \textit{relational product} of $\mathsf{LP}$ and $\mathsf{K}_3$ (in the sense of \cite{Maddux2006}): this means that when $\Gamma$ $\mathsf{ST}$-entails $\Delta$, there exists a formula $\phi$ such that $\Gamma$ $\mathsf{K}_3$-entails $\phi$, and $\phi$ $\mathsf{LP}$-entails $\Delta$. Similarly, we prove that the set of $\mathsf{TS}$-valid inferences can be identified using a dual notion, namely as the \textit{relational sum} of the sets of $\mathsf{LP}$-valid and $\mathsf{K}_3$-valid inferences: this means that when $\Gamma$ $\mathsf{TS}$-entails $\Delta$, then every formula $\phi$ is such that either $\Gamma$ $\mathsf{LP}$-entails $\phi$, or $\phi$ $\mathsf{K}_3$-entails $\Delta$. These results give us a more precise articulation of the intuition that $\mathsf{ST}$ and $\mathsf{TS}$ are hybrids of $\mathsf{K}_3$ and $\mathsf{LP}$ (viz. \cite{Cobrerosetal2020, zardini2022final}).

To prove our results, we start by rehearsing the definitions of validity in $\mathsf{K}_3$, $\mathsf{LP}$, $\mathsf{ST}$ and $\mathsf{TS}$ in Section \ref{sec:definitions}. In Section \ref{sec:relations}, we briefly review why intersection and union fail to adequately characterize $\mathsf{ST}$ and $\mathsf{TS}$, and we introduce the operations of relational product and relation sum that underlie our main result. Sections \ref{sec:st} and \ref{sec:ts} give the proofs of the main claims. Section \ref{sec:converse} shows that order matters in these identities, that is $\mathsf{ST}$ and $\mathsf{TS}$ differ from the logics obtained when swapping $\mathsf{K}_3$ and $\mathsf{LP}$ as the operands of sum and product. Section \ref{sec:duality} defines a notion of duality between logics that relates $\mathsf{ST}$ to $\mathsf{TS}$ in exactly the sense in which $\mathsf{LP}$ relates to $\mathsf{K}_3$, this notion is then used to characterize extensionally $\mathsf{ST}$ and $\mathsf{TS}$ in terms of either $\mathsf{LP}$ or $\mathsf{K}_3$.  In section \ref{sec:discussion}, finally, we close with some comparisons: first we examine the way in which the relational characterization of $\mathsf{ST}$ can be related to a refinement of the interpolation theorem proved by \cite{milne2016refinement} for classical logic in a trivalent setting. Then we compare our characterization of the notion of duality of logic with those of \cite{cobreros2021cant} and \cite{DaRéetal2020}, and show that the notion we have got gives us an invariant characterization that was missing.
Finally, we conclude with a statement of some open problems in order to indicate the fruitfulness of our approach.


\section{$\mathsf{LP}$, $\mathsf{K}_3$, $\mathsf{ST}$ and $\mathsf{TS}$ 
}\label{sec:definitions} 
This section provides some technical preliminaries needed for our results. We define the basic propositional logics of interest in this paper, namely $\mathsf{K}_3$, $\mathsf{LP}$, $\mathsf{TS}$, and $\mathsf{ST}$, and the concepts of satisfiability and validity needed to establish our main results.

\begin{defn}[Language]\label{language2}
The propositional language $\mathcal{L}$ is built from a denumerably infinite set $\textit{Var} = \lbrace p, p', \ldots\rbrace$ of propositional variables, using the nullary constants \textit{verum} ($\top$), \textit{falsum} ($\bot$), \textit{lambda} ($\lambda$), negation ($\neg$), disjunction ($\lor$) and conjunction ($\wedge$).
\end{defn}

\noindent
Elements of $\mathcal{L}$ will be denoted by lower case Greek letters or upper case Latin letters, depending on the context. Subsets of $\mathcal{L}$ will be denoted by upper case Greek letters.

\begin{defn}[Inference]
An inference on $\mathcal{L}$ is an ordered pair $\langle \Gamma, \Delta \rangle$ (hereafter symbolized by $\Gamma \Rightarrow \Delta$) where $\Gamma, \Delta \subseteq \mathcal{L}$ are finite sets. We denote by $\text{\textit{INF}}(\mathcal{L})$ the set of all inferences on $\mathcal{L}$.
\end{defn}

\begin{defn}[SK-valuation]
A Strong-Kleene valuation $v$ (or SK-valuation for short) is a function from propositional variables to truth-values:
\begin{align*}
v: \textit{Var} \longrightarrow \left\lbrace 0, \sfrac{1}{2}, 1 \right\rbrace
\end{align*}
that is extended to the whole language by letting
\begin{align*}
v^*(p) &= v(p),\\
v^*(\top) &= 1,\\
v^*(\bot) &= 0,\\
v^*(\lambda) &= \sfrac{1}{2},\\
v^*(\neg \phi) &= 1-v^*(\phi),\\
v^*(\phi \lor \psi) &= max(v^*(\phi), v^*(\psi)),\\
v^*(\phi \wedge \psi) &= min(v^*(\phi), v^*(\psi)).
\end{align*}
\end{defn}

Strong Kleene valuations permit the definition of two notions of truth for a propositional formula, namely strict truth and tolerant truth \parencite[]{cobreros2015vagueness}: a formula $\phi$ is said to be strictly true if $v(\phi) \in \lbrace 1 \rbrace$, and tolerantly true if $v(\phi) \in \lbrace \sfrac{1}{2}, 1\rbrace$. By extension, we will say that a formula is strictly satisfiable, or tolerantly satisfiable, if and only if  there is a valuation that makes it strictly true, or tolerantly true.

We get two well-known infra-classical logics from those definitions, depending on whether logical consequence is defined as the preservation of strict truth or of tolerant truth, corresponding to Strong Kleene logic $\mathsf{K}_3$ \parencite{Kleene1952-KLEITM}, and to the Logic of Paradox $\mathsf{LP}$ \parencite{Asenjo1966, Priest1979}, respectively. Each preserves a different set of designated values from the premises to the conclusions. While the logical consequence of $\mathsf{K}_3$ preserves the value in $S = \lbrace 1 \rbrace$, the one of $\mathsf{LP}$ preserves the values in $T = \lbrace \sfrac{1}{2}, 1 \rbrace$.

\begin{defn}[$\mathsf{LP}$-satisfaction, $\mathsf{LP}$-validity]
An SK-valuation $v$ satisfies an inference $\Gamma \Rightarrow \Delta$ in $\mathsf{LP}$ (symbolized by $v \models_{\mathsf{LP}} \Gamma \Rightarrow \Delta$) if and only if  $v(\gamma) \neq 0$ for all $\gamma \in \Gamma$ only if $v(\delta) \neq 0$ for some $\delta \in \Delta$. An inference $ \Gamma \Rightarrow \Delta$ is valid in $\mathsf{LP}$ (symbolized by $\models_{\mathsf{LP}} \Gamma \Rightarrow \Delta$) if and only if $v \models_{\mathsf{LP}} \Gamma \Rightarrow \Delta$ for all SK-valuations $v$.
\end{defn}

\begin{defn}[$\mathsf{K}_3$-satisfaction, $\mathsf{K}_3$-validity]
An SK-valuation $v$ satisfies an inference $\Gamma \Rightarrow \Delta$ in $\mathsf{K}_3$ (symbolized by $v \models_{\mathsf{K}_3} \Gamma \Rightarrow \Delta$) if and only if  $v(\gamma)=1$ for all $\gamma \in \Gamma$ only if $v (\delta)=1$ for some $\delta \in \Delta$. An inference $ \Gamma \Rightarrow \Delta$ is valid in $\mathsf{K}_3$ (symbolized by $\models_{\mathsf{K}_3} \Gamma \Rightarrow \Delta$) if and only if $v \models_{\mathsf{K}_3} \Gamma \Rightarrow \Delta$ for all SK-valuations $v$.
\end{defn}

It can easily be checked that $\mathsf{K}_3$ is paracomplete in the sense that $\not\models_{\mathsf{K}_3} \emptyset \Rightarrow \phi \lor \neg \phi$, and that $\mathsf{LP}$ is paraconsistent in the sense that $ \not\models_{\mathsf{LP}} \phi \wedge \neg \phi \Rightarrow \emptyset$.

The two definitions of logical consequence for $\mathsf{K}_3$ and $\mathsf{LP}$ rely on the standard definition of a logical consequence as necessary preservation of the same set of designated values from premises to conclusions. Following \textcite{zardini2008model}, \textcite{Cobrerosetal2012, cobreros2015vagueness} consider a more liberal definition of logical consequence. Rather than moving from premises to conclusions from one set of designated values to the same set, they consider the idea of moving from one set of designated values to another set, different from the first. This idea gives rise to two other logics: $\mathsf{ST}$ and $\mathsf{TS}$. As indicated by the names, $\mathsf{ST}$'s logical consequence involves moving from strict to tolerant truth, while $\mathsf{TS}$'s involves moving from tolerant truth to strict truth.\footnote{Both logics make earlier appearances, for $\mathsf{TS}$ in the work of \cite{Malinowski1990} (on q-consequence), and for $\mathsf{ST}$ in that of \cite{Girard1987} and of \cite{Frankowski2004} (on the cut-free sequent calculus for $\mathsf{LK}$, and on p-consequence, respectively). Even earlier manifestations of $\mathsf{ST}$-consequence can be found in Strawson and Belnap (see \cite{BarrioEgre2022} for a historical overview).}

\begin{defn}[$\mathsf{ST}$-satisfaction, $\mathsf{ST}$-validity]
An SK-valuation $v$ satisfies an inference $\Gamma \Rightarrow \Delta$ in $\mathsf{ST}$ (symbolized by $v \models_{\mathsf{ST}} \Gamma \Rightarrow \Delta$) if and only if  $v(\gamma)=1$ for all $\gamma \in \Gamma$ only if $v(\delta) \neq 0$ for some $\delta \in \Delta$. An inference $ \Gamma \Rightarrow \Delta$ is valid in $\mathsf{ST}$ (symbolized by $\models_{\mathsf{ST}} \Gamma \Rightarrow \Delta$) if and only if $v \models_{\mathsf{ST}} \Gamma \Rightarrow \Delta$ for all SK-valuations $v$.
\end{defn}

\begin{defn}[$\mathsf{TS}$-satisfaction, $\mathsf{TS}$-validity]
An SK-valuation $v$ satisfies an inference $\Gamma \Rightarrow \Delta$ in $\mathsf{TS}$ (symbolized by $v \models_{\mathsf{TS}} \Gamma \Rightarrow \Delta$) if and only if  $v (\gamma) \neq 0$ for all $\gamma \in \Gamma$ only if $v(\delta)=1$ for some $\delta \in \Delta$. An inference $ \Gamma \Rightarrow \Delta$ is valid in $\mathsf{TS}$ (symbolized by $\models_{\mathsf{TS}} \Gamma \Rightarrow \Delta$) if and only if $v \models_{\mathsf{TS}} \Gamma \Rightarrow \Delta$ for all SK-valuations $v$.
\end{defn}

It follows from the definition of $\mathsf{ST}$-validity that the set of valid inferences of the $\lambda$-free fragment of $\mathsf{ST}$ coincides exactly with the set of inferences valid in classical propositional logic (\cite{Cobrerosetal2012}): an inference is valid in the $\lambda$-free fragment of $\mathsf{ST}$ if and only if it is classically valid. The $\lambda$-including fragment of $\mathsf{ST}$ is non-classical, however, in particular it fails structural transitivity, because although $\top \Rightarrow \lambda$ and $\lambda \Rightarrow \bot$ are both $\mathsf{ST}$-valid, $\top \Rightarrow \bot$ is not. $\mathsf{TS}$ on the other hand fails structural reflexivity, already on the $\lambda$-free fragment: $p\Rightarrow p$ is not valid. Moreover, $\mathsf{TS}$ is empty of validities in the fragment not involving $\top$ and $\bot$, and its validities in the language including both constants are very limited (more on this below).

Hereinafter, we will distinguish between a logic $\mathsf{L}$, its set of valid inferences $\mathbb{L}^{+}$, and its set of antivalid inferences $\mathbb{L}^{-}$. 
    The reason for this distinction is twofold. Firstly, as exemplified by the difference between classical logic and $\mathsf{ST}$ regarding transitivity, two logics can share the same set of valid inferences but still differ at the metainferential level.\footnote{\label{fn:local}At least on the local definition of metainferential validity, and not on the global one. For the distinction between these two notions, which lies beyond the scope of this paper, see  \textcite{DicherPaoli2019} and \textcite{Barrioetal2020}.} Secondly, as shown by \textcite{scambler2020classical}, two logics can have the same valid inferences and metainferences while having different antivalidities.\footnote{Several notions of antivalidities have been put forward in the literature on $\mathsf{ST}$. \textcite{scambler2020classical} and \textcite{DaRéetal2020} say that an inference $\Gamma \Rightarrow \Delta$ is antivalid when for all $v$,  $v \not\models \Gamma \Rightarrow \Delta$, while \textcite{cobreros2021cant} say that an inference $\Gamma \Rightarrow \Delta$ is antivalid when for all $v$,  if $v(\gamma)$ is not designated for all $\gamma \in \Gamma$, then $v(\delta)$ is not designated for some $\delta \in \Delta$. It can easily be checked that given any of these notions, classical logic and $\mathsf{ST}$ fail to have the same antivalidities. In Section \ref{sec:duality}, we will make use of the second notion.} Hence, we will reserve the notation $\mathsf{L}$ --- for what we call a logic --- to refer to what fixes the satisfaction (respectively antisatisfaction) and validity (respectively antivalidity) conditions of an argument over a language, and we will use the notations $\mathbb{L}^{+}$ and $\mathbb{L}^{-}$ to denote a set of ordered pairs of sentences of the language, determined by either the validity or the antivalidity conditions of $\mathsf{L}$, respectively. Informally, however, we will allow ourselves to say that the logic $\mathsf{ST}$ is the product of $\mathsf{K}_3$ and $\mathsf{LP}$, to mean the more accurate claim that $\mathbb{ST}^+$ is the product of $\mathbb{K}_3^{+}$ and $\mathbb{LP}^{+}$. 
    
    We won't need to introduce the notion of antivalidity until section \ref{sec:duality}. First, however, we will need the notion of a set of valid inferences:


\begin{defn}[Set of valid inferences]
\begin{align*}
\mathbb{L}^{+} &= \lbrace \Gamma \Rightarrow \Delta  : \ \models_{\mathsf{L}} \Gamma \Rightarrow \Delta \rbrace
\end{align*}
\end{defn}

\begin{figure}
\[
 \begin{tikzpicture}
\draw (0,4) node (st) {$\mathbb{ST}^{+}$};
\draw (2,2) node (ss) {$\mathbb{K}_3^{+}$};
\draw (-2,2) node (tt) {$\mathbb{LP}^{+}$};
\draw (0,0) node (ts) {$\mathbb{TS}^{+}$};
 \draw[-] (ts) -- (ss);
 \draw[-] (ts) -- (tt);
 \draw[-] (ss) -- (st);
 \draw[-] (tt) -- (st);
\end{tikzpicture}
\]
\caption{Relations between $\mathbb{TS}^{+}$, $\mathbb{K}_3^{+}$, $\mathbb{LP}^{+}$ and $\mathbb{ST}^{+}$}\label{fig:incl}
\end{figure}

Inclusion relations between the set of validities of the above four logics are well-known and depicted in Figure \ref{fig:incl} (\cite{Cobrerosetal2012, Chemlaetal2017}): $\mathbb{TS}^{+}$ is properly included in $\mathbb{LP}^{+}$ and in $\mathbb{K}_3^{+}$, and both of those sets are incomparable with each other but are included in $\mathbb{ST}^{+}$. What is yet to be found concerns the way in which $\mathbb{ST}^{+}$ and $\mathbb{TS}^{+}$ can be characterized from $\mathbb{K}_3^{+}$ and $\mathbb{LP}^{+}$.

\section{Relational Sum and Product}\label{sec:relations}
In order to introduce the operations that allow us to define $\ST$ and $\TS$ in terms of $\K$ and $\LP$, we start by first rehearsing why neither of those systems coincides with the intersection or the union of both logics. While it is well-known that $\ST$ is not the intersection of $\K$ and $\LP$, it is tempting to think that $\ST$ might be the union of $\K$ and $\LP$, but in fact it is not (see the references given in \cite{wintein2016all} regarding the misidentification of classical logic with $\K\cup\LP$). It is useful, therefore, to review counterexamples to both identities.

\begin{fact}
$\TS$ and $\ST$ are neither the intersection, nor the union of $\K$ and $\LP$.
\end{fact}

\begin{proof}
    For the intersection: $p \Rightarrow p$ is not in $\TS$, but it is both in $\LP$ and in $\K$. Conversely, $\emptyset \Rightarrow \neg p \lor p$ is in $\ST$, but is is not in $\K$, so it is not in $\K\cap \LP$ either. Hence, $\K\cap\LP$ is equal neither to $\TS$ nor to $\ST$.
    
Regarding the union: since both $\mathbb{K}_3^{+}$ and $\mathbb{LP}^{+}$ are reflexive, it follows that $\TS$, which is not reflexive, cannot be the union. For $\ST$, the following inference --- first brought to our notice in 2010 by David Ripley, and featuring in e.g. \cite[p.~511]{wintein2016all}---, gives a counterexample:

$$p \lor (q \wedge \neg q) \Rightarrow p \wedge (q \lor \neg q)$$

\noindent
This inference is valid in $\mathsf{ST}$, but it is neither valid in $\mathsf{K}_3$ nor valid in $\mathsf{LP}$. Let $v$ and $v'$ be two SK-valuations such that $v(p) = 1$, $v'(p) = 0$ and $v(q)= v'(q)= \sfrac{1}{2}$. Then $v(p \lor (q \wedge \neg q)) = 1$, $v(p \wedge (q \lor \neg q)) = \sfrac{1}{2}$, $v'(p \lor (q \wedge \neg q)) = \sfrac{1}{2}$ and $v'(p \wedge (q \lor \neg q)) = 0$. So, $v \not\models_{\mathsf{K}_3} p \lor (q \wedge \neg q) \Rightarrow p \wedge (q \lor \neg q)$ and $v' \not\models_{\mathsf{LP}} p \lor (q \wedge \neg q) \Rightarrow p \wedge (q \lor \neg q)$.
\end{proof}

One may nonetheless observe that $\models_{\mathsf{K}_3} p \lor (q \wedge \neg q) \Rightarrow p$ and $\models_{\mathsf{LP}} p \Rightarrow p \wedge (q \lor \neg q)$. This suggests that $p$ acts as a connecting formula. More generally, therefore, let the relational composition of $\K$ and $\LP$ be defined as follows:

\begin{defn}[Relational composition of $\K$ and $\LP$]
\begin{align*}
\mathbb{K}_3^{+}\ \circ\ \mathbb{LP}^{+} = \lbrace \langle \Gamma, \Delta \rangle : \langle \Gamma, \Theta\rangle \in \mathbb{K}_3^{+} \ \text{and} \ \langle \Theta, \Delta\rangle \in \mathbb{LP}^{+} \ \text{for some} \ \Theta \rbrace.
\end{align*}
\end{defn}

\noindent Relational composition is strictly more inclusive than the union:

\begin{fact}
    $\mathbb{K}_3^{+} \cup \mathbb{LP}^{+} \subsetneq\mathbb{K}_3^{+} \circ \mathbb{LP}^{+}$.
\end{fact}
\begin{proof}
Assume first that $\Gamma, \Delta \neq \emptyset$. If $\Gamma \Rightarrow \Delta \in \mathbb{K}_3^{+}$, given that $\Delta \Rightarrow \Delta \in \mathbb{LP}^{+}$, then $\Gamma \Rightarrow \Delta \in \mathbb{K}_3^{+} \circ \mathbb{LP}^{+}$. If $\Gamma \Rightarrow \Delta \in \mathbb{LP}^{+}$, given that $\Gamma \Rightarrow \Gamma \in \mathbb{K}_3^{+}$, then $\Gamma \Rightarrow \Delta \in \mathbb{K}_3^{+} \circ \mathbb{LP}^{+}$. Assume now that $\Gamma = \emptyset$. If $\emptyset \Rightarrow \Delta \in \mathbb{K}_3^{+}$, then $\Delta \neq \emptyset$, so $\Delta \Rightarrow \Delta \in \mathbb{LP}^{+}$, and hence $\emptyset \Rightarrow \Delta \in \mathbb{K}_3^{+} \circ \mathbb{LP}^{+}$. Now, if $\emptyset \Rightarrow \Delta \in \mathbb{LP}^{+}$, $\top \Rightarrow \Delta \in \mathbb{LP}^{+}$, and given that $\emptyset \Rightarrow \top \in \mathbb{K}_3^{+}$, we obtain $\emptyset \Rightarrow \Delta \in \mathbb{K}_3^{+} \circ \mathbb{LP}^{+}$. Finally, assume $\Delta = \emptyset$. If $\Gamma \Rightarrow \emptyset \in \mathbb{K}_3^{+}$, $\Gamma \Rightarrow \bot \in \mathbb{K}_3^{+}$, and given that $\bot \Rightarrow \emptyset \in \mathbb{LP}^{+}$, we obtain $\Gamma \Rightarrow \emptyset \in \mathbb{K}_3^{+} \circ \mathbb{LP}^{+}$. If $\Gamma \Rightarrow \emptyset \in \mathbb{LP}^{+}$, then $\Gamma \neq \emptyset$, so $\Gamma \Rightarrow \Gamma \in \mathbb{K}_3^{+}$, and hence $\Gamma \Rightarrow \emptyset \in \mathbb{K}_3^{+} \circ \mathbb{LP}^{+}$.

To prove proper inclusion, note that the previous formula $p \lor (q \wedge \neg q) \Rightarrow p \wedge (q \lor \neg q)$ belongs to $\K\circ \LP$.
\end{proof}

However, as defined relation composition is still inadequate to exactly match $\ST$, an observation originally due to Pablo Cobreros:

\begin{fact}
    $\ST\neq \K\circ\LP$.
\end{fact}

\begin{proof}  
$\models_{\mathsf{K}_3} p \lor \neg p \Rightarrow p, \neg p$ and $\models_{\mathsf{LP}} p, \neg p \Rightarrow p \wedge \neg p$. Hence, $p \lor \neg p \Rightarrow p \wedge \neg p \in \mathbb{K}_3^{+} \circ \mathbb{LP}^{+}$ but $p \lor \neg p \Rightarrow p \wedge \neg p \not\in \mathbb{ST}^{+}$.
\end{proof}

The problematic case arises because composition is defined over sets of formulae, and on our definitions of validity, the meaning of a comma on the left-hand side of the sequent is different than its meaning on its right-hand side. While on the left-hand side it must be read as a conjunction, on the right-hand side it must be read as a disjunction. In order to rule out this problematic case, we need to restrict the middle-term of the composition to a singleton set, that is to a single formula. 

To that end, we introduce the following enhanced definition of \textit{relative product} and \textit{relative sum} (see \cite{Maddux2006} for an exposition, and \cite{Peirce1883} for the original definitions):

\begin{defn}[Relative product]\label{DefRelP}
Let $R$ and $S$ be two arbitrary binary relations on a set $U$ and $V \subseteq U$, then
\begin{align*}
R \hspace{0.1em} \stackunder{$\mid$}{$\ms V$} \hspace{0.1em} S &= \lbrace \langle x, z \rangle : \langle x, y \rangle \in R \ \text{and} \ \langle y, z \rangle \in S \ \text{for some} \ y \in V \rbrace
\end{align*}
\end{defn}

\begin{defn}[Relative sum]\label{DefRelS}
Let $R$ and $S$ be two arbitrary binary relations on a set $U$ and $V \subseteq U$, then
\begin{align*}
R \hspace{0.2em} \stackunder{$\dag$}{$\ms V$} \hspace{0.2em} S &= \lbrace \langle x, z \rangle : \langle x, y \rangle \in R \ \text{or} \ \langle y, z \rangle \in S \ \text{for all} \ y \in V \rbrace
\end{align*}
\end{defn}

The duality between these two notions is expressed by the following de Morgan identities:

\begin{fact}[Duality of product and sum]\label{duality}
\begin{align*}
R \hspace{0.1em} \stackunder{$\mid$}{$\ms V$} \hspace{0.1em} S &= \overline{\overline{R} \hspace{0.2em} \stackunder{$\dag$}{$\ms V$} \hspace{0.2em} \overline{S}}\\
R \hspace{0.2em} \stackunder{$\dag$}{$\ms V$} \hspace{0.2em} S &= \overline{\overline{R} \hspace{0.1em} \stackunder{$\mid$}{$\ms V$} \hspace{0.1em} \overline{S}}
\end{align*}
\end{fact}

\begin{proof}
   \noindent\begin{paracol}{2}
\begin{align*}
&\langle x, z\rangle \in R \hspace{0.1em} \stackunder{$\mid$}{$\ms V$} \hspace{0.1em} S\\
\text{iff} \ &\textup{there is} \ y \in V \ \textup{s.t.} \langle x, y\rangle \in R\\
&\text{and} \ \langle y, z\rangle \in S \\
\text{iff} \ &\textup{there is} \ y \in V \ \textup{s.t.} \langle x, y\rangle \not\in \overline{R}\\
&\text{and} \ \langle y, z\rangle \not\in \overline{S} \\
\text{iff} \ &\textup{not for all} \ y \in V, \langle x, y\rangle \in \overline{R}\\
&\text{or} \ \langle y, z\rangle \in \overline{S} \\
\text{iff} \ &\langle x, z\rangle \not\in \overline{R} \hspace{0.2em} \stackunder{$\dag$}{$\ms V$} \hspace{0.2em} \overline{S}\\
\text{iff} \ &\langle x, z\rangle \in \overline{\overline{R} \hspace{0.2em} \stackunder{$\dag$}{$\ms V$} \hspace{0.2em} \overline{S}}
\end{align*}
\switchcolumn[1]
\begin{align*}
&\langle x, z\rangle \in R  \hspace{0.2em} \stackunder{$\dag$}{$\ms V$} \hspace{0.2em} S\\
\text{iff} \ &\textup{for all} \ y \in V,\langle x, y\rangle \in R\\
&\text{or} \ \langle y, z\rangle \in S \\
\text{iff} \ &\textup{for all} \ y \in V, \langle x, y\rangle \not\in \overline{R}\\
&\text{or} \ \langle y, z\rangle \not\in \overline{S} \\
\text{iff} \ &\textup{for no} \ y \in V, \langle x, y\rangle \in \overline{R}\\
&\text{and} \ \langle y, z\rangle \in \overline{S} \\
\text{iff} \ &\langle x, z\rangle \not\in \overline{R} \hspace{0.1em} \stackunder{$\mid$}{$\ms V$} \hspace{0.1em} \overline{S}\\
\text{iff} \ &\langle x, z\rangle \in \overline{\overline{R} \hspace{0.1em} \stackunder{$\mid$}{$\ms V$} \hspace{0.1em} \overline{S}}
\end{align*}
\end{paracol}
\end{proof}

With these definitions at our disposal, we can provide an exact characterization of $\ST$ and $\TS$ as product and sum. We will use $\mathcal{P}(\mathcal{L})$ as $U$ and $\lbrace \lbrace \phi \rbrace : \phi \in \mathcal{L} \rbrace$ as $V$ (omitting from now on its mention under the symbols for relative multiplication and relative sum). The main theorems we shall prove are:

\begin{theorem}\label{thm:st}
   $ \ST=\K\mid\LP$
\end{theorem}

\begin{theorem}\label{thm:ts}
    $\TS=\LP\ \dag\ \K$
\end{theorem}


\noindent
We proceed to prove Theorem \ref{thm:st} in Section \ref{sec:st} and Theorem \ref{thm:ts} in Section \ref{sec:ts}.

\section{$\mathbb{ST}^{+} = \mathbb{K}_3^{+} \mid \mathbb{LP}^{+}$}\label{sec:st}
In order to prove the equality expressed by Theorem \ref{thm:st}, we proceed in two steps, proving first the left inclusion and then the right inclusion. For the left inclusion, we show how to construct a connecting formula $\phi$ such that if $\models_{\mathsf{ST}} \Gamma \Rightarrow \Delta$, then $\models_{\mathsf{K}_3} \Gamma \Rightarrow \phi$ and $\models_{\mathsf{LP}} \phi \Rightarrow \Delta$. For the right inclusion, we will see that under the assumption that $\models_{\mathsf{K}_3} \Gamma \Rightarrow \phi$ and $\models_{\mathsf{LP}} \phi \Rightarrow \Delta$ for some $\phi \in \mathcal{L}$, the $\mathsf{ST}$-validity of $\Gamma \Rightarrow \Delta$ follows almost immediately from the definition of $\mathsf{K}_3$-validity and $\mathsf{LP}$-validity.

We start by defining the function $\textup{At}$ that allows to isolate the atoms and the constants of a formula:

\begin{defn}[Set of atoms of a formula]
The function $\text{At}: \mathcal{L} \longrightarrow \mathcal{P}(\textit{Var})$ is a recursively defined function such that:

\noindent
\begin{minipage}[t]{.4\textwidth}
\begin{align*}
&\text{At}(p) = \lbrace p \rbrace,\\
&\text{At}(\top) = \lbrace \top \rbrace,\\
&\text{At}(\bot) = \lbrace \bot \rbrace,\\
\end{align*}
\end{minipage}
\begin{minipage}[t]{.4\textwidth}
\begin{align*}
&\text{At}(\lambda) = \lbrace \lambda \rbrace,\\
&\text{At}( \neg \phi) = \text{At}(\phi),\\
&\text{At}( \phi \lor \psi) = \text{At}(\phi \wedge \psi) = \text{At}(\phi) \cup \text{At}(\psi).
\end{align*}
\end{minipage}
\end{defn}

\noindent
If $\Gamma \subseteq \mathcal{L}$, $\text{At}(\Gamma)$ abbreviates $\bigcup \lbrace \text{At}(\gamma) : \gamma \in \Gamma \rbrace$.\medskip

We then introduce the notion of \textit{partial sharpening} adapted from the notion of \textit{sharpening}, found in \cite{scambler2020classical}. Intuitively, given an SK-valuation $v$, an SK-valuation $v^*$ qualifies as a sharpening of $v$ if and only if it agrees with $v$ on all the atoms of the language that take a classical value, and possibly disagrees with $v$ on the others. The notion of partial sharpening of a valuation with respect to a set of atoms extends the definition of a sharpening to any subset of the set of atoms of the language.

\begin{defn}[Partial sharpening]\label{D1}
Let $v$ be an SK-valuation. An SK-valuation $v^*$ is said to be a partial sharpening of $v$ with respect to a set of atoms $\Sigma \subseteq \textit{Var} \cup \lbrace \top, \bot, \lambda \rbrace$ if and only if for all $\alpha \in \Sigma$, if $v(\alpha) \neq \sfrac{1}{2}$, then $v^*(\alpha) = v(\alpha)$.
\end{defn}

The next two corollaries describe useful properties of the notion of partial sharpening.

\begin{corollary}\label{C1}
If $v^*$ is a partial sharpening of $v$ with respect to $\Sigma$, then $v^*$ is a partial sharpening of $v$ with respect to any $\Theta \subseteq \Sigma$.
\end{corollary}
\begin{proof} Note that if the condition holds for all $\alpha_1 \in \Sigma$, it also holds for all $\alpha_2 \in \Theta \subseteq \Sigma$.
\end{proof}

\begin{corollary}\label{C2}
If $v^*$ is a partial sharpening of $v$ with respect to $\Sigma$ and $\Theta$, then $v^*$ is a partial sharpening of $v$ with respect to $\Sigma \cup \Theta$.
\end{corollary}
\begin{proof} Note that if the condition holds for all $\alpha_1 \in \Sigma$ and for all $\alpha_2 \in \Theta$, it also holds for all $\alpha \in \Sigma \cup \Theta$.
\end{proof}

We then turn to the first lemma of importance, which illustrates the monotonicity of the SK-scheme (see \cite{dare2023three}).\footnote{
Given the order $\le_{\textup{I}}$ defined on the set $\lbrace 0, \sfrac{1}{2}, 1 \rbrace$ so that $\sfrac{1}{2} <_{\textup{I}} 0$ and $\sfrac{1}{2} <_{\textup{I}} 1$, a scheme is said to be \textit{monotonic} if each of its $n$-ary operations is order-preserving with respect to the Cartesian order $\le_{\textup{I}}^{n}$ defined by $\langle x_1, \ldots, x_n \rangle \le_{\textup{I}}^{n} \langle y_1, \ldots, y_n \rangle \iff (i = 1, \ldots, n) \ x_i \le_{\textup{I}} y_i$.}
Given a formula $\phi$, if $\phi$ is attributed a classical value by a valuation $v$, any valuation $v^*$ that agrees with $v$ on the classical values attributed to the atoms of $\phi$ will also agree with $v$ on the value of $\phi$.

\begin{lemma}\label{L1}
For all SK-valuations $v$, if $v(\phi) \neq \sfrac{1}{2}$, then for all partial sharpenings $v^*$ of $v$ with respect to $\textup{At}(\phi)$, it holds that $v^*(\phi) = v(\phi)$.
\end{lemma}

\begin{proof}
    The proof is by simple induction on the complexity of $\phi$. To prove the inductive step we rely on Corollary \ref{C1} to show that if $v^*$ is a partial sharpening of $v$ with respect to $\textup{At}(\psi \lor \chi)$ or $\textup{At}(\psi \wedge \chi)$, it is also a partial sharpening of $v$ with respect to $\textup{At}(\psi)$ and $\textup{At}(\chi)$, allowing thus the use of the inductive hypothesis.
\end{proof}

For any $\mathsf{ST}$-valid inference $\Gamma \Rightarrow \Delta$ such that all the $\gamma \in \Gamma$ jointly take the value $1$ for some valuation $v$, we show how to construct a term $\phi$ such that $\Gamma \Rightarrow \phi$ is $\mathsf{K}_3$-valid and $\phi \Rightarrow \Delta$ is $\mathsf{LP}$-valid. This term $\phi$ will be the $\mathsf{K}_3$ disjunctive normal form of $\bigwedge \Gamma$. The next two definitions show how to construct such a term.

\begin{defn}[$\Gamma,v$-conjunction]\label{DefConj} Let $\Gamma \subseteq \mathcal{L}$ be a finite nonempty set of formulae. If $v$ is such that $v(\gamma)=1$ for every $\gamma\in \Gamma$, let $C_{\Gamma,v} \coloneqq \bigwedge_{v(\alpha) \neq \sfrac{1}{2}} \alpha^{\sim}$, with $\alpha \in \textup{At}(\Gamma)$ and $\alpha^\sim=\alpha$ if $v(\alpha)=1$ and $\alpha^\sim=\neg \alpha$ if $v(\alpha)=0$.
\end{defn}

In other words, if there is a valuation $v$ that strictly satisfies each $\gamma \in \Gamma$, we construct the conjunction of all the literals $\alpha^\sim$ -- built from the atoms of $\Gamma$ -- that take value $1$ according to $v$. That $C_{\Gamma,v}$ is well-defined when $v$ is such that $v(\gamma)=1$ for every $\gamma\in \Gamma$ follows from the fact that some atom $\alpha$ of the formulae in $\Gamma$ must be such that $v(\alpha)=1$ or $v(\alpha)=0$, for otherwise we would have $v(\gamma)=\sfrac{1}{2}$ for every $\gamma\in \Gamma$.

\begin{defn}\label{DefDNF}
    Given a finite nonempty set of formulae $\Gamma$, let $D_\Gamma \coloneqq \bigvee_{v(\bigwedge \Gamma)=1} C_{\Gamma,v}$ when there is $v$ such that $v(\bigwedge\Gamma)=1$, and let $D_\Gamma=\bot$ otherwise.
\end{defn}

The first clause of the definition tells us how to construct the $\mathsf{K}_3$ disjunctive normal form of $\bigwedge \Gamma$ from each $C_{\Gamma,v}$, while the last clause enable us to deal with the case where no such $C_{\Gamma,v}$ exists, namely when there is no $v$ such that $v(\bigwedge\Gamma)=1$. 

If the clause of Definition \ref{DefConj} is satisfied, $C_{\Gamma,v}$ necessarily takes the value $1$. This fact is illustrated by the following lemma:

\begin{lemma}

Let $\Gamma$ be a nonempty set of formulae. For all $v$, if $v(\bigwedge \Gamma)=1$, then $v(C_{\Gamma,v})=1$.
\end{lemma}

\begin{proof}
 If $v(\bigwedge \Gamma)=1$, then $v$ is such that $v(\gamma)=1$ for every $\gamma\in \Gamma$. So $C_{\Gamma,v}$ is defined. By construction,  and by the SK evaluation of a conjunction, $v(C_{\Gamma,v})=1$.
\end{proof}


The next lemma shows that the disjunction $D_{\Gamma}$ of all the $C$s admissible with respect to $\Gamma$ is $\mathsf{K}_3$-equivalent to $\bigwedge \Gamma$.

\begin{lemma}\label{Lab1}
Let $\Gamma$ be a nonempty set of formulae. Then
\begin{align*}
\models_{\mathsf{K}_3} \Gamma \Rightarrow D_{\Gamma} \ \text{and} \ \models_{\mathsf{K}_3} D_{\Gamma} \Rightarrow \bigwedge \Gamma.
\end{align*}
\end{lemma}

\begin{proof}
    The case for which there is no valuation making the formulae in $\Gamma$ strictly true together is obvious, since then $D_{\Gamma}=\bot$.

    So consider the case in which the formulae in $\Gamma$ can be made strictly true jointly. If $v(\bigwedge \Gamma)=1$ then $v(C_{\Gamma, v})=1$ by the previous lemma, and $v(D_{\Gamma})=1$. So $\models_{\mathsf{K}_3} \Gamma \Rightarrow D_{\Gamma}$.

    If there is $v'$ such that $v'(D_{\Gamma})=1$, then $v'(\bigvee_{v(\bigwedge \Gamma)=1} C_{\Gamma,v})=1$. So for some $v$ such that $v(\bigwedge \Gamma)=1$, we have $v'(C_{\Gamma,v})=1$. Then $v'$ must agree on the atoms of $\Gamma$ to which $v$ gives a classical value, and possibly differs on the atoms of $\Gamma$ to which $v$ gives the value $\sfrac{1}{2}$. Thus, $v'$ is a partial sharpening of $v$ with respect to $\textup{At}(\bigwedge\Gamma)$, and therefore since $v(\bigwedge \Gamma)=1$, we also have $v'(\bigwedge \Gamma)=1$ by Lemma \ref{L1}.
\end{proof}

We are now in a position to prove the main theorem of the section. Namely that $\models_{\mathsf{ST}} \Gamma \Rightarrow \Delta$ if and only if there is $\phi \in \mathcal{L}$ such that $\models_{\mathsf{K}_3} \Gamma \Rightarrow \phi$ and $\models_{\mathsf{LP}} \phi \Rightarrow \Delta$.

\begin{theorem}\label{ThST}
$\mathbb{ST}^{+} = \mathbb{K}_3^{+} \mid \mathbb{LP}^{+}$
\end{theorem}

\noindent
\begin{proof}

($\subseteq$) Assume that $\Gamma \Rightarrow \Delta \in \mathbb{ST}^{+}$, then we prove that there is a connecting formula $\phi$ such that $\Gamma \Rightarrow \phi \in \mathbb{K}_3^{+}$ and $\phi \Rightarrow \Delta \in \mathbb{LP}^{+}$.


    When $\Gamma$ is empty, we let $\phi:=\top$. Clearly, $\models_{\mathsf{K}_3} \Gamma \Rightarrow \top$, and since $\models_{\mathsf{ST}} \emptyset \Rightarrow \Delta$ by assumption, for all $v$, $v(\delta)\neq 0$ for some $\delta \in \Delta$, and therefore $\models_{\mathsf{LP}} \top \Rightarrow \Delta$.

    When $\Gamma$ is nonempty, we let $\phi:=D_{\Gamma}$. $\models_{\mathsf{K}_3} \Gamma \Rightarrow D_{\Gamma}$ follows from Lemma \ref{Lab1}. So let us prove $\models_{\mathsf{LP}} D_{\Gamma} \Rightarrow \Delta$.

    Assume $\not\models_{\mathsf{LP}} D_{\Gamma} \Rightarrow \Delta$ for the sake of contradiction. It follows that there is $v$ such that $v(D_{\Gamma})\neq 0$ and $v(\delta) = 0$ for all $\delta \in \Delta$. If $v(D_{\Gamma}) = \sfrac{1}{2}$, there is $v'$ such that $v(C_{\Gamma, v'}) = \sfrac{1}{2}$. So, for all conjuncts $\alpha^\sim$ of $C_{\Gamma, v'}$, $v(\alpha^\sim) \neq 0$. Consider $v^*$ which is exactly like $v$ except that for all the conjuncts $\alpha^\sim$ of $C_{\Gamma, v'}$, if $v(\alpha^\sim) = \sfrac{1}{2}$, $v^*(\alpha^\sim) = 1$. Then $v^*(C_{\Gamma, v'}) = 1$, so $v^*(D_{\Gamma}) = 1$. Moreover, $v^*$ is a partial sharpening of $v$ with respect to $\textup{At}(C_{\Gamma, v'})$ by construction, and is exactly like $v$ for all $\beta \in \textup{At}(\Delta) - \textup{At}(C_{\Gamma, v'})$, so trivially it is also a partial sharpening of $v$ with respect to $\textup{At}(\Delta) - \textup{At}(C_{\Gamma, v'})$. Therefore, by Corollary \ref{C2}, $v^*$ is a partial sharpening of $v$ with respect to $\textup{At}(C_{\Gamma, v'}) \cup (\textup{At}(\Delta) - \textup{At}(C_{\Gamma, v'}))$, and by Corollary \ref{C1} to $\textup{At}(\Delta) \subseteq \textup{At}(C_{\Gamma, v'}) \cup (\textup{At}(\Delta) - \textup{At}(C_{\Gamma, v'}))$. Now, given that $v(\delta) = 0$ for all $\delta \in \Delta$, by Lemma \ref{L1}, $v^*(\delta) = 0$ for all $\delta \in \Delta$. But by Lemma \ref{Lab1}, $\models_{\mathsf{K}_3} D_{\Gamma} \Rightarrow \bigwedge \Gamma$, so $v^*(\bigwedge \Gamma)=1$, and thus $v^* \not\models_{\mathsf{ST}} \Gamma \Rightarrow \Delta$, which is impossible. If $v(D_{\Gamma}) = 1$, then $v(\bigwedge\Gamma) = 1$ by the same lemma, and $v(\delta) = 0$ for all $\delta \in \Delta$ by assumption, which again contradicts $\models_{\mathsf{ST}} \Gamma \Rightarrow \Delta$.

    ($\supseteq$) Assume that $\Gamma \Rightarrow \Delta \in \mathbb{K}_3^{+} \mid \mathbb{LP}^{+}$. Then there is $\phi \in \mathcal{L}$ such that $\models_{\mathsf{K}_3} \Gamma \Rightarrow \phi$ and $\models_{\mathsf{LP}} \phi \Rightarrow \Delta$. So, for all $v$ if $v(\gamma) = 1$ for all $\gamma \in \Gamma$, then $v(\phi)=1$, and if $v(\phi) \neq 0$, then  $v(\delta) \neq 0$ for some $\delta \in \Delta$. Hence, for all $v$ if $v(\gamma) = 1$ for all $\gamma \in \Gamma$, then $v(\delta) \neq 0$ for some $\delta \in \Delta$, thus $\models_{\mathsf{ST}} \Gamma \Rightarrow \Delta$, and we have $\Gamma \Rightarrow \Delta \in \mathbb{ST}^{+}$, as desired.
\end{proof}

To illustrate the theorem, let us go back to the case of the $\mathsf{ST}$-valid argument $p \lor (q \wedge \neg q) \Rightarrow p \wedge (q \lor \neg q)$. The theorem implies that there is a formula $D$ such that $p \lor (q \wedge \neg q) \Rightarrow D \in \mathbb{K}_3^{+}$ and $D\Rightarrow  p \wedge (q \lor \neg q)\in \mathbb{LP}^{+}$. In this case, the method described produces $D=(p \wedge q)\vee (p \wedge \neg q) \vee p$,  and one can check that $p \lor (q \wedge \neg q) \Rightarrow (p \wedge q)\vee (p \wedge \neg q) \vee p \in \mathbb{K}_3^{+}$ and $(p \wedge q)\vee (p \wedge \neg q) \vee p\Rightarrow  p \wedge (q \lor \neg q)\in \mathbb{LP}^{+}$. This formula is more complex than the atomic formula $p$ which we saw above would suffice, but we note that it too may be described as an interpolant in the sense of Craig's interpolation theorem, since it only contains atoms common to $\Gamma$ and $\Delta$. We return to this connection to interpolation in section \ref{sec:discussion}.


\section{$\TS = \mathbb{LP}^{+} \ \dag \ \mathbb{K}_3^{+}$}\label{sec:ts}
Turning now to the proof of the second theorem, we start by proving a simple but useful lemma showing that for any formula $\phi$, if $\phi$ is tolerantly satisfiable, then the SK-valuation $v$ that maps each propositional variable to $\sfrac{1}{2}$ also makes $\phi$ tolerantly satisfiable, and if $\phi$ is strictly falsifiable, then the same $v$ makes $\phi$ strictly falsifiable too.

\begin{lemma}\label{L2}
Let $v'$ be such that for all $p \in \textit{Var}$, $v'(p)= \sfrac{1}{2}$. Then for all $\phi \in \mathcal{L}$,
\begin{enumerate}[label=(\roman*)]
\item\label{point1} If there is $v$ such that $v(\phi) \neq 0$, then $v'(\phi) \neq 0$.
\item\label{point2} If there is $v$ such that $v(\phi) \neq 1$, then $v'(\phi) \neq 1$.
\end{enumerate}
\end{lemma}

\begin{proof}
    \ref{point1} and \ref{point2} are proved simultaneously by induction on the complexity of $\phi$.
\end{proof}

The next lemma shows a sense in which $\mathsf{TS}$-validity requires either that one of its premise is tolerantly unsatisfiable or that one of its conclusion is strictly unfalsifiable:

\begin{lemma}\label{Lem12}
If $\models_\mathsf{TS} \Gamma \Rightarrow \Delta$, then
\begin{enumerate}[label=(\roman*)]
\item For some $\gamma \in \Gamma$, it holds for all $v$ that $v(\gamma) = 0$
\item[or]
\item for some $\delta \in \Delta$, it holds for all $v$ that $v(\delta) = 1$.
\end{enumerate}
\end{lemma}

\noindent
\textit{Proof.} We prove the contrapositive. Assume that for all $\gamma \in \Gamma$ there is $v$ such that $v(\gamma) \neq 0$ and for all $\delta \in \Delta$, there is $v$ such that $v(\delta) \neq 1$. Let $v'$ be such that $v'(p) = \sfrac{1}{2}$ for all $p \in \textit{Var}$. By Lemma \ref{L2}, for all $\gamma \in \Gamma, v'(\gamma) \neq 0$ and for all $\delta \in \Delta$, $v'(\delta) \neq 1$. Hence, $\not\models_\mathsf{TS} \Gamma \Rightarrow \Delta$.\qed

We then turn to the proof of the main theorem:

\begin{theorem}
$\mathbb{TS}^{+} = \mathbb{LP}^{+} \ \dag \ \mathbb{K}_3^{+}$
\end{theorem}

\noindent
\textit{Proof.}

($\subseteq$) Assume $\Gamma \Rightarrow \Delta \in \mathbb{TS}^{+}$. Then $\models_\mathsf{TS} \Gamma \Rightarrow \Delta$, and by Lemma \ref{Lem12}, for some $\gamma \in \Gamma$, it holds for all $v$ that $v(\gamma) = 0$ or for some $\delta \in \Delta$, it holds for all $v$ that $v(\delta) = 1$. So, for all $\phi \in \mathcal{L}$, $\models_{\mathsf{LP}} \Gamma \Rightarrow \phi$ or $\models_{\mathsf{K}_3} \phi \Rightarrow \Delta$, and therefore $\Gamma \Rightarrow \Delta \in \mathbb{LP}^{+} \ \dag \ \mathbb{K}_3^{+}$.

($\supseteq$) We prove the contrapositive. Assume $\Gamma \Rightarrow \Delta \not\in \mathbb{TS}^{+}$. Then $\not\models_{\mathsf{TS}} \Gamma \Rightarrow \Delta$. So, there is a valuation $v$ such that $v(\gamma) \neq 0$ for all $\gamma \in \Gamma$ and $v(\delta) \neq 1$ for all $\delta \in \Delta$. Let $p \in \textit{Var} - (\text{At}(\Gamma) \cup \text{At}(\Delta))$. Since $\textit{Var}$ is an infinite set and both $\text{At}(\Gamma)$ and $\text{At}(\Delta)$ are finite sets, there is such a $p$. Let $v'$ be exactly like $v$ except that if $v(p) \neq 0$, $v'(p)=0$. Then $v'(\gamma) \neq 0$ for all $\gamma \in \Gamma$ since $p$ does not appear in $\Gamma$, and $v'(p) = 0$, so $v' \not\models_{\mathsf{LP}} \Gamma \Rightarrow p$. Now, let $v''$ be exactly like $v$ except that if $v(p) \neq 1$, $v''(p)=1$. Then $v''(p)= 1$ and $v''(\delta)\neq 1$ for all $\delta \in \Delta$ since $p$ does not appear in $\Delta$, so $v'' \not\models_{\mathsf{K}_3} p \Rightarrow \Delta$. Hence, there is $\phi \coloneqq p$ such that $\not\models_{\mathsf{LP}} \Gamma \Rightarrow \phi$ and $\not\models_{\mathsf{K}_3} \phi \Rightarrow \Delta$, meaning that $\Gamma \Rightarrow \Delta \not\in \mathbb{LP}^{+} \ \dag \ \mathbb{K}_3^{+}$. \qed \bigskip

As we shall see in the next section, this characterization shows that $\mathsf{TS}$ can be conceptualized in terms of the theorems of $\mathsf{K}_3$, and of the antitheorems of $\mathsf{LP}$. The theorems of $\mathsf{K}_3$ indeed consist only of the trivial inferences involving $\top$ in the consequent (e.g. $\Gamma \Rightarrow \psi \lor \top$) and the antitheorems of $\mathsf{LP}$ only of the trivial inferences involving $\bot$ in the antecedent (e.g. $\phi \wedge \bot \Rightarrow \Delta$).

\section{The case of $\mathbb{LP}^{+} \mid \mathbb{K}_3^{+}$ and $\mathbb{K}_3^{+} \ \dag \ \mathbb{LP}^{+}$}\label{sec:converse}


How tight are the previous characterizations of $\mathbb{ST}^{+}$ and $\mathbb{TS}^{+}$? Could $\mathbb{LP}^{+}$ and $\mathbb{K}_3^{+}$ be swapped in the statement of either theorem? We show that the answer is negative.


Let us consider $\mathbb{TS}^{+}$ first. The relative sum of $\mathbb{K}_3^{+}$ and $\mathbb{LP}^{+}$ is not equal to $\mathbb{TS}^{+}$, but to the set of antitheorems and theorems of $\mathsf{ST}$. In order to prove this claim, we first define the antitheorems and the theorems of a logic, and then show that the operation of relative sum extracts the antitheorems of the first operand and the theorems of the second. After rehearsing a well-known fact about $\mathsf{K}_3$ and $\mathsf{LP}$, namely that the first has the same antitheorems as $\mathsf{ST}$ and the second the same theorems as $\mathsf{ST}$, we prove our main claim.

\begin{defn}[Theorem and antitheorem]
Given a logic $\mathsf{L}$, a set of formulae $\Gamma$ is said to be a theorem of $\mathsf{L}$ when $\models_{\mathsf{L}} \emptyset \Rightarrow \Gamma$, and an antitheorem of $\mathsf{L}$ when $\models_{\mathsf{L}} \Gamma \Rightarrow \emptyset$.
\end{defn}

\begin{defn}[Trivial theorem and antitheorem]
    A theorem or antitheorem $\Gamma$ is trivial if there is $\gamma \in \Gamma$ such that for all $v, v'$: $v(\gamma)=v'(\gamma)$.
\end{defn}


We now focus on two well-known facts about antitheorems and theorems. For the sake of simplicity, we will only consider finitary logics, namely logics with a language in which formulas are constructed in a finite number of steps and where inferences are defined as pairs of finite sets of formulas.\footnote{When $\mathsf{L}$ is an infinitary logic, \ref{Ati2} of Fact \ref{fct:ath} does not necessarily entail \ref{Ati5}. To see this, note that for any $p \in \textit{Var}$, $p \in \textup{At}(q \wedge \neg q \wedge \bigwedge \textit{Var})$, so \ref{Ati5} is false, yet $q \wedge \neg q \wedge \bigwedge \textit{Var}$ is for instance an antitheorem of the classical infinitary logic $\mathcal{L}_{\omega_1}$ permitting conjunctions and disjunctions of length $<\omega_1$ and with $|\textit{Var}|<\omega_1$. To remedy this, one can restate \ref{Ati5} as $\models_{\mathsf{L}}\sigma_p(\Gamma) \Rightarrow p$ with $\sigma_p$ a substitution that renames the variables in such a way that $p$ does not appear in the antecedent anymore (see \cite{prenosil_2023}). Since $p \not\in \sigma_p(q \wedge \neg q \wedge \bigwedge \textit{Var})$ and $\sigma_p(q \wedge \neg q \wedge \bigwedge \textit{Var}) \Rightarrow p$ is valid in $\mathcal{L}_{\omega_1}$, the preceding case does not constitute a counterexample to the entailment of \ref{Ati5} by \ref{Ati2} anymore.}


\begin{fact}\label{fct:ath}
    Let $\mathsf{L}$ be a finitary logic with a consequence relation defined from the pair of sets of designated values $\langle D_1, D_2 \rangle$. The following are equivalent:
    \begin{enumerate}[label=(\roman*), align=left]
    \setlength\itemsep{0pt}
        \item\label{Ati2} $\Gamma$ is an antitheorem of $\mathsf{L}$.
        \item\label{Ati3} $\models_{\mathsf{L}} \Gamma \Rightarrow \Delta$ for all $\Delta \subseteq \mathcal{L}$.
        \item\label{Ati4} $\models_{\mathsf{L}} \Gamma \Rightarrow \phi$ for all $\phi \in \mathcal{L}$.
        \item\label{Ati5} $\models_{\mathsf{L}} \Gamma \Rightarrow p$ for some $p \not\in \textup{At}(\Gamma)$.
    \end{enumerate}
\end{fact}

\begin{proof}
    That \ref{Ati3} follows from \ref{Ati2}, \ref{Ati4} from \ref{Ati3}, and \ref{Ati5} from \ref{Ati4} is straightforward to prove, so we just show that \ref{Ati2} follows from \ref{Ati5}. Assume that $\not\models_{\mathsf{L}} \Gamma \Rightarrow \emptyset$. Then there is $v$ such that $v(\gamma) \in D_1$ for all $\gamma \in \Gamma$. Let $p \not\in \textup{At}(\Gamma)$, since $\mathsf{L}$ is finitary and $\Gamma$ finite, such a $p$ exists. Assume further that $v'$ is exactly like $v$ except that if $v(p) \in D_2$, $v'(p) \not\in D_2$. Then $v'(\gamma) \in D_1$ for all $\gamma \in \Gamma$ since $p \not\in \textup{At}(\Gamma)$ and hence $\not\models_{\mathsf{L}} \Gamma \Rightarrow p$ for some $p \not\in \textup{At}(\Gamma)$. 
\end{proof}


\begin{fact}\label{fct:th}
    Let $\mathsf{L}$ be a finitary logic with a consequence relation defined from the pair of sets of designated values $\langle D_1, D_2 \rangle$. The following are equivalent:
    \begin{enumerate}[label=(\roman*), align=left]
    \setlength\itemsep{0pt}
        \item\label{Ati2bis} $\Delta$ is a theorem of $\mathsf{L}$.
        \item\label{Ati3bis} $\models_{\mathsf{L}} \Gamma \Rightarrow \Delta$ for all $\Gamma \subseteq \mathcal{L}$.
        \item\label{Ati4bis} $\models_{\mathsf{L}} \phi \Rightarrow \Delta$ for all $\phi \in \mathcal{L}$.
        \item\label{Ati5bis} $\models_{\mathsf{L}} p \Rightarrow \Delta$ for some $p \not\in \textup{At}(\Delta)$.
    \end{enumerate}
\end{fact}

\begin{proof}
    The proof is symmetric to the proof of Fact \ref{fct:ath}.
\end{proof}

Given Fact \ref{fct:ath}, we identify the set of antitheorems of a logic with the set of all the valid inferences of the form $\Gamma \Rightarrow \Delta$ with $\Gamma$ an antitheorem, and similarly for the set of theorems given Fact \ref{fct:th}. 

\begin{defn}[Sets of theorems and antitheorems]
The sets of theorems and antitheorems of a logic $\mathsf{L}$ are defined respectively as follows:
\begin{align*}
    \mathbb{L}^t &\coloneqq \lbrace \Gamma \Rightarrow \Delta \in \textit{INF}(\mathcal{L}) : \ \models_{\mathsf{L}} \emptyset \Rightarrow \Delta \rbrace,\\
    \mathbb{L}^a &\coloneqq \lbrace \Gamma \Rightarrow \Delta \in \textit{INF}(\mathcal{L}) : \ \models_{\mathsf{L}} \Gamma \Rightarrow \emptyset \rbrace.
\end{align*}
\end{defn}

The next fact shows that the operation of relative sum extracts the antitheorems of the first operand and the theorems of the second.

\begin{fact}\label{fct:sumat}
Given $\mathsf{L}_1$ and $\mathsf{L}_2$ two finitary logics with a consequence relation defined from a pair of sets of designated values:
\begin{align*}
    \mathbb{L}_1^{+} \ \dag \ \mathbb{L}_2^{+} = \mathbb{L}^a_1 \cup \mathbb{L}^t_2.
\end{align*}
\end{fact}

\begin{proof}
($\subseteq$) If $\Gamma \Rightarrow \Delta \in \mathbb{L}_1^{+} \ \dag \ \mathbb{L}_2^{+}$, then for all $\phi \in \mathcal{L}$, $\models_{\mathsf{L}_1} \Gamma \Rightarrow \phi$ or $\models_{\mathsf{L}_2} \phi \Rightarrow \Delta$. Since $\mathsf{L}_1$, $\mathsf{L}_2$ are finitary and $\Gamma, \Delta$ are finite, there is $p \in \textit{Var}$ such that $p \not\in \textup{At}(\Gamma) \cup \textup{At}(\Delta)$ and $\models_{\mathsf{L}_1} \Gamma \Rightarrow p$ or $\models_{\mathsf{L}_2} p \Rightarrow \Delta$. So $\models_{\mathsf{L}_1} \Gamma \Rightarrow \Delta$ or $\models_{\mathsf{L}_2} \Gamma \Rightarrow \Delta$ by Facts \ref{fct:ath} and \ref{fct:th}, and therefore $\Gamma \Rightarrow \Delta \in \mathbb{L}^a_1 \cup \mathbb{L}^t_2$.
    
    ($\supseteq$) Let $\Gamma \Rightarrow \Delta \in \mathbb{L}^a_1 \cup \mathbb{L}_2^t$. If $\Gamma \Rightarrow \Delta \in \mathbb{L}^a_1$, then $\models_{\mathsf{L}_1} \Gamma \Rightarrow \phi$ for all $\phi \in \mathcal{L}$ by Fact \ref{fct:ath}. If $\Gamma \Rightarrow \Delta \in \mathbb{L}_2^t$, then $\models_{\mathsf{L}_2} \psi \Rightarrow \Delta$ for all $\psi \in \mathcal{L}$ by Fact \ref{fct:th}. In both cases we obtain $\models_{\mathsf{L}_1} \Gamma \Rightarrow \phi$ for all $\phi \in \mathcal{L}$ or $\models_{\mathsf{L}_2} \psi \Rightarrow \Delta$ for all $\psi \in \mathcal{L}$, which entails in turn $\models_{\mathsf{L}_1} \Gamma \Rightarrow \phi$ or $\models_{\mathsf{L}_2} \phi \Rightarrow \Delta$ for all $\phi \in \mathcal{L}$. And therefore $\Gamma \Rightarrow \Delta \in \mathbb{L}_1^{+} \ \dag \ \mathbb{L}_2^{+}$.
\end{proof}

We then go over a well-known fact about $\mathsf{K}_3$ and $\mathsf{LP}$. Namely that $\mathsf{K}_3$ has the same antitheorems as $\mathsf{ST}$, and $\mathsf{LP}$ the same theorems.

\begin{fact}\label{fct:STat}
    $\mathbb{K}^a_3 = \mathbb{ST}^a$ and $\mathbb{LP}^t = \mathbb{ST}^t$.
\end{fact}

\begin{proof}
Note on the one hand that $\models_{\mathsf{K}_3} \Gamma \Rightarrow \emptyset$ iff for all $v$, $v(\gamma) \neq 1$ for some $\gamma \in \Gamma$ iff $\models_{\mathsf{ST}} \Gamma \Rightarrow \emptyset$, and on the other that $\models_{\mathsf{LP}} \emptyset \Rightarrow \Delta$ iff for all $v$, $v(\delta) \neq 0$ for some $\delta \in \Delta$ iff $\models_{\mathsf{ST}} \emptyset \Rightarrow \Delta$.
\end{proof}

This fact enables us to characterize the relative sum of $\mathbb{K}_3^{+}$ and $\mathbb{LP}^{+}$ as the set of antitheorems and theorems of $\mathsf{ST}$.

\begin{theorem}
$\mathbb{K}_3^{+} \ \dag \ \mathbb{LP}^{+} = \mathbb{K}^a_3 \cup \mathbb{LP}^t = \mathbb{ST}^a \cup \mathbb{ST}^t$.
\end{theorem}

\begin{proof}
    From Facts \ref{fct:sumat} and \ref{fct:STat}.
\end{proof}

It follows that the relative sum of $\mathbb{K}_3^{+}$ and $\mathbb{LP}^{+}$ is not equal to $\mathbb{TS}^{+}$ since for instance $\emptyset \Rightarrow p \lor \neg p$ is an element of $\mathbb{ST}^a \cup \mathbb{ST}^t$ but not of $\mathbb{TS}^{+}$, and similarly for $p \wedge \neg p \Rightarrow \emptyset$. 

This result provides in addition a better understanding of the role played by the relative sum in the characterization of $\mathbb{TS}^{+}$. Just as $\mathbb{K}_3^{+} \ \dag \ \mathbb{LP}^{+}$ is the set of antitheorems of $\mathsf{K}_3$ and of theorems of $\mathsf{LP}$ -- that is the set of antitheorems and theorems of $\mathsf{ST}$, as shown by the previous fact --, $\mathbb{TS}^{+} = \mathbb{LP}^{+} \ \dag \ \mathbb{K}_3^{+}$ is the set consisting of the antitheorems of $\mathbb{LP}^{+}$ and the theorems of $\mathbb{K}_3^{+}$. Since it is known that $\mathsf{LP}$ has no non-trivial antitheorems and $\mathsf{K}_3$ no non-trivial theorems, $\mathbb{TS}^{+}$ is therefore the set of trivial antitheorems of $\mathsf{LP}$ and trivial theorems of $\mathsf{K}_3$.

 
Turning to the case of $\mathsf{ST}$, we show that in the presence of $\lambda$, the relative product of the sets of $\mathsf{LP}$-valid and $\mathsf{K}_3$-valid inferences is the set of valid inferences of the trivial, universal logic, in which anything follows from anything.

\begin{fact}
    $\mathbb{ST}^{+}\neq \mathbb{LP}^{+} \mid \mathbb{K}_3^{+} = \mathcal{P}(\mathcal{L}) \times \mathcal{P}(\mathcal{L})$
\end{fact}
\begin{proof}
Note that for any $\Gamma, \Delta \subseteq \mathcal{L}$, $\models_{\mathsf{LP}}\Gamma \Rightarrow \lambda$ and $\models_{\mathsf{K}_3}\lambda \Rightarrow \Delta$. Hence, for any $\Gamma, \Delta $, $\Gamma \Rightarrow \Delta \in \mathbb{LP}^{+} \mid \mathbb{K}_3^{+}$.
\end{proof}

However, we can prove that $\mathbb{LP}^{+} \mid \mathbb{K}_3^{+} = \mathbb{ST}^{+}$ when the constant $\lambda$ is removed from the language. The connecting formula will be constructed by taking the conjunction of the premises with the classical tautologies built from each atom of the conclusions. Doing so forces one of the conclusions to take a classical value whenever the connecting formula is $\mathsf{K}_3$-satisfied.

\begin{theorem} Let $\mathcal{L}^{-}$ be the language defined  
in Definition \ref{language2} without the constant $\lambda$, then
\begin{align*}
    \mathbb{ST}^{+} = \mathbb{LP}^{+} \mid \mathbb{K}_3^{+}
\end{align*}
\end{theorem}

\begin{proof}
($\subseteq$) Assume $\Gamma \Rightarrow \Delta \in \mathbb{ST}^{+}$, that is $\models_{\mathsf{ST}} \Gamma \Rightarrow \Delta$. 

We start with the limit cases. If there is no valuation $v$ such that $v(\gamma) \neq 0$ for all $\gamma \in \Gamma$, we set $\phi \coloneqq \bot$, since $\models_{\mathsf{LP}} \Gamma \Rightarrow \bot$ and $\models_{\mathsf{K}_3} \bot \Rightarrow \Delta$. If for all $v$, $v(\delta)= 1$ for some $\delta \in \Delta$, we set $\phi \coloneqq \top$, since $\models_{\mathsf{LP}} \Gamma \Rightarrow \top$ and $\models_{\mathsf{K}_3} \top \Rightarrow \Delta$. 

\noindent
We now turn to the main case. Let 
\begin{align*}
C \coloneqq \bigwedge \lbrace (\alpha \lor \neg \alpha) : \alpha \in \text{At}(\Delta) \rbrace
\end{align*}

\noindent
and $D \coloneqq \bigwedge \Gamma \wedge C$.

\noindent
We first show that $\models_{\mathsf{LP}} \Gamma \Rightarrow D$. Assume $v(\gamma) \in \lbrace \sfrac{1}{2}, 1 \rbrace$ for all $\gamma \in \Gamma$. For all $p \in \textit{Var}$, $v(p \lor \neg p) \neq 0$, so $v(C) \neq 0$, and since $v(\gamma) \neq 0$ for all $\gamma \in \Gamma$ by assumption, $v(D) \neq 0$. Hence, $\models_{\mathsf{LP}} \Gamma \Rightarrow D$.

\noindent
Assume now that $\not\models_{\mathsf{K}_3} D \Rightarrow \Delta$. Then there is $v$ such that $v(D) = 1$ and $v(\delta) \neq 1$ for all $\delta \in \Delta$. Given that $v(D) = 1$, $v(\gamma) = 1$ for all $\gamma \in \Gamma$ and $v(C) = 1$. From the latter, it follows that for all $\alpha \in \text{At}(\Delta)$, $v(\alpha\lor \neg \alpha) = 1$, therefore for all $\alpha \in \text{At}(\Delta)$, $v(\alpha) \neq \sfrac{1}{2}$. So, $v(\delta) \neq \sfrac{1}{2}$ for all $\delta \in \Delta$, meaning that $v(\delta) = 0$ for all $\delta \in \Delta$. Hence, $v \not\models_{\mathsf{ST}} \Gamma \Rightarrow \Delta$, which contradicts our first assumption. Therefore, $\models_{\mathsf{K}_3} D \Rightarrow \Delta$. Finally, since there is $\phi \coloneqq D$ such that $\models_{\mathsf{LP}} \Gamma \Rightarrow \phi$ and $\models_{\mathsf{K}_3} \phi \Rightarrow \Delta$, $\Gamma \Rightarrow \Delta \in \mathbb{LP}^{+} \mid \mathbb{K}_3^{+}$.

($\supseteq$) Assume $\Gamma \Rightarrow \Delta \in \mathbb{LP}^{+} \mid \mathbb{K}_3^{+}$ but $\Gamma \Rightarrow \Delta \not\in \mathbb{ST}^{+}$. From the latter, it follows that $\not\models_{\mathsf{ST}} \Gamma \Rightarrow \Delta$. Hence, there is $v$ such that $v(\gamma)=1$ for all $\gamma \in \Gamma$ and $v(\delta)=0$ for all $\delta \in \Delta$. From the former, it follows that there is $\phi$ such that $\models_{\mathsf{LP}} \Gamma \Rightarrow \phi$ and $\models_{\mathsf{K}_3} \phi \Rightarrow \Delta$. 
If $v(\phi) = 0$, then $\not\models_{\mathsf{LP}} \Gamma \Rightarrow \phi$, if $v(\phi) = 1$, then $\not\models_{\mathsf{K}_3} \phi \Rightarrow \Delta$, therefore $v(\phi) = \sfrac{1}{2}$. Now, let $v'$ be exactly like $v$ except that for all $p \in \textit{Var}$ such that $v(p) = \sfrac{1}{2}$, $v'(p) \neq \sfrac{1}{2}$. Then $v'(\gamma)=1$ for all $\gamma \in \Gamma$ and $v'(\delta)=0$ for all $\delta \in \Delta$ by Lemma \ref{L1}, and $v'(\phi) = 1$ or $v'(\phi) = 0$ since all its atoms get a classical value. And again, $\not\models_{\mathsf{LP}} \Gamma \Rightarrow \phi$ or $\not\models_{\mathsf{K}_3} \phi \Rightarrow \Delta$, which is inconsistent with our first assumption. So, $\Gamma \Rightarrow \Delta \in \mathbb{ST}^{+}$, as desired.
\end{proof}

The upshot is that $\mathbb{LP}^+$ and $\mathbb{K}_{3}^{+}$ are not interchangeable in the definitions of $\mathbb{ST}^+$ and $\mathbb{TS}^{+}$ using relational product and sum. The symmetry found for $\mathbb{ST}^+$ in the $\lambda$-free fragment is special, and does not function for $\mathbb{TS}^+$. This may reflect the fact that $\mathbb{TS}^+$ is already non-reflexive in the $\lambda$-free fragment, whereas $\mathbb{ST}^+$ is only nontransitive in the $\lambda$-including fragment (under the notion of global notion of metainferential validity, see fn. see \ref{fn:local}).




\section{On the duality between $\textsf{ST}$ and $\mathsf{TS}$}\label{sec:duality}

The results of the previous section tell us that $\mathbb{ST}^+$ and $\mathbb{TS}^+$ are both definable in terms of $\mathbb{K}_3^{+}$ and $\mathbb{LP}^+$, from the dual operations of sum and product. This suggests that $\mathbb{ST}^+$ and $\mathbb{TS}^+$ are dual logics. Whether $\mathbb{ST}^+$ and $\mathbb{TS}^+$ can be seen as such is moot, however, and depends on how duality between logics is defined.
\cite{cobreros2021cant} distinguish two notions of duality, which they call \textit{operational} duality and \textit{structural} duality, following Gentzen's distinction between operational rules and structural rules for a logic, and they show that $\mathbb{ST}^+$ and $\mathbb{TS}^+$ are structural duals, but not operational duals, and conversely for $\mathbb{K}_3^{+}$ and $\mathbb{LP}^+$. In this section, we propose a revision of the notion of operational duality allowing us to establish that by combining the two notions of operational and structural duality, $\mathbb{ST}^+$ and $\mathbb{TS}^+$ stand in a relation that is exactly the same as the relation to which $\mathbb{K}_3^{+}$ and $\mathbb{LP}^+$ stand to each other, and similarly for their counterparts defined in terms of antivalidity. We proceed in three steps: we first introduce our modified notion of operational duality, then Cobreros et al.'s notion of structural duality, finally we use the apparatus of relative sum and product to provide general interdefinability results for $\mathbb{ST}^+$, $\mathbb{ST}^-$, $\mathbb{TS}^+$ and $\mathbb{TS}^-$ in terms of the validities and antivalidities of either $\mathsf{LP}$ or $\mathsf{K}_3$.

\subsection{Operational duality, revised}

The notion of operational duality used by \cite{cobreros2021cant} originates in \cite{Cobrerosetal2012}, and for the sake of clarity we propose to call it negation duality. It has the following definition:

\begin{defn}[Negation duality]\label{NegDua}
    A logic $\mathsf{L}_1$ is \textit{negation dual} to a logic $\mathsf{L}_2$ when
   \[\begin{array}{ccc}
        \models_{\mathsf{L}_1} \Gamma \Rightarrow \Delta & \textup{iff} & \models_{\mathsf{L}_2} \neg (\Delta) \Rightarrow \neg (\Gamma)\\
    \end{array}\]
with $\neg (\Gamma) = \lbrace \neg \gamma : \gamma \in \Gamma \rbrace$ for all $\Gamma \subseteq \mathcal{L}$. 
\end{defn}

The logics $\mathsf{K}_3$ and $\mathsf{LP}$ are  \textit{negation} duals of each other. Concomitantly, it is known that $\mathsf{ST}$ and $\mathsf{TS}$ are \textit{negation} self-dual (see e.g. \cite{Cobrerosetal2012, DaRéetal2020}). 

Despite its intrinsic interest, such a notion does not license a straightforward extensional interdefinability of $\mathsf{K}_3$ and $\mathsf{LP}$. This can easily be noticed by letting $n: \mathcal{P}(\textit{INF}(\mathcal{L})) \longrightarrow \mathcal{P}(\textit{INF}(\mathcal{L}))$ be a mapping such that $n(\mathbb{X}) = \lbrace \neg (\Delta) \Rightarrow \neg (\Gamma) : \Gamma \Rightarrow \Delta \in \mathbb{X} \rbrace$. Clearly $\mathbb{K}_3^{+} \neq n(\mathbb{LP}^{+})$, since not all $\mathsf{K}_3$-valid inferences are of the form $\neg(\Delta) \Rightarrow \neg(\Gamma)$.

Because of that, we propose a distinct notion of operational duality, which will enable us to prove the extensional interdefinability between $\mathsf{K}_3$ and $\mathsf{LP}$. This notion of duality can be found in \cite{Kleene1952-KLEITM} (and goes back to \cite{Schroder1877}), where it is proven that for any two formulae $\phi$, $\psi$, if the inference $\phi \Rightarrow \psi$ is classically valid, then so is the inference $\psi' \Rightarrow \phi'$, with $\phi'$ and $\psi'$ obtained from $\phi$ and $\psi$ by interchanging $\wedge$ with $\lor$. In our case, we define a mapping $\Mysim$ that will substitute $\top$ and $\bot$ for each other in all formulae, and similarly $\lor$ and $\wedge$.


\begin{defn}
    Let $\Mysim: \mathcal{L} \to \mathcal{L}$ be a mapping such that:
    \begin{enumerate}
     \setlength\itemsep{0pt}
        \item[--] $\Mysim p = p$, if $p \in \textit{Var}$,
        \item[--] $\Mysim \top = \bot$,
        \item[--] $\Mysim \bot = \top$,
        \item[--] $\Mysim \lambda = \lambda$,
        \item[--] $\Mysim\neg \phi=\neg \Mysim\phi$,
        \item[--] $\Mysim(\phi \wedge \psi) = \Mysim\phi \lor \Mysim\psi$,
        \item[--] $\Mysim(\phi \lor \psi) = \Mysim\phi \wedge \Mysim\psi$.
    \end{enumerate}
\end{defn}
\noindent
The function $\Mysim$ is extended to all $\Gamma \subseteq \mathcal{L}$ by $\Mysim(\Gamma) \coloneqq \lbrace \Mysim\gamma : \gamma \in \Gamma \rbrace$, and to all $\mathbb{X} \subseteq \textit{INF}(\mathcal{L})$ by $\Mysim(\mathbb{X}) \coloneqq \lbrace \Mysim(\Gamma) \Rightarrow \Mysim(\Delta) : \Gamma \Rightarrow \Delta \in \mathbb{X} \rbrace$.


For instance, this function applied to the formula $p \wedge q$ will produce the formula $p \lor q$, and applied to the formula $p \wedge (q \lor \neg q)$ will produce the formula $p \lor (q \wedge \neg q)$. We use the mapping $\Mysim$ for the following revision of the notion of operational duality:

\begin{defn}[Operational duality for formulas, and inferences]
    For any two formulae $\phi$ and $\psi$, $\psi$ is called the (revised) operational dual of $\phi$ when $\Mysim\phi = \psi$. For any two inferences $\Gamma \Rightarrow \Delta$ and $\Delta' \Rightarrow \Gamma'$, 
$\Delta' \Rightarrow \Gamma'$ is called the operational dual of $\Gamma \Rightarrow \Delta$ when $\Mysim\Gamma = \Gamma'$ and $\Mysim\Delta = \Delta'$.
\end{defn}

The next two lemmas express useful properties of $\Mysim$, namely that it is involutive and commutes with the relational inverse function. 

\begin{lemma}\label{lem:inv}
For all $\phi \in \mathcal{L}$, $\Mysim\Mysim \phi = \phi$.
\end{lemma}

\begin{proof}
    The proof is by induction on the length of $\phi$.
\end{proof}






\begin{lemma}\label{lem:asso}
    For all $\mathbb{X} \subseteq \textit{INF}(\mathcal{L})$, $(\Mysim \mathbb{X})^{-1}= \Mysim (\mathbb{X}^{-1})$.
\end{lemma}

\begin{proof}\leavevmode
    \begin{align*}
        \Gamma \Rightarrow \Delta \in (\Mysim \mathbb{X})^{-1}\quad &\textup{iff} \quad \Delta \Rightarrow \Gamma \in \Mysim\mathbb{X}\\
        &\textup{iff} \quad \Mysim\Delta \Rightarrow \Mysim\Gamma \in \Mysim\Mysim\mathbb{X}=\mathbb{X} \ \textup{by Lemma \ref{lem:inv}}\\
        &\textup{iff} \quad \Mysim\Gamma \Rightarrow \Mysim\Delta \in \mathbb{X}^{-1}\\
        &\textup{iff} \quad \Mysim\Mysim\Gamma \Rightarrow \Mysim\Mysim\Delta \in \Mysim(\mathbb{X}^{-1})\\
        &\textup{iff} \quad \Gamma \Rightarrow \Delta \in \Mysim(\mathbb{X}^{-1}) \ \textup{by Lemma \ref{lem:inv}}.
    \end{align*}
\end{proof}

To show that $\mathsf{K}_3$ and $\mathsf{LP}$ are extensionally interdefinable, we will have to prove first that they are operationally dual. The next lemma describes the operational duality between two formulae in terms of dual valuations. Two SK valuations $v$ and $v'$ are said to be dual when for all $p \in \textit{Var}$, $v'(p)= 1 - v(p)$. The dual of a valuation $v$ will be noted $\val$.

\begin{lemma}
   For all valuations $v$ and all $\phi \in \mathcal{L}$, $v(\phi) = 0$ if and only if $v^{\star}(\Mysim\phi) = 1$ and $v(\phi) = 1$ if and only if $v^{\star}(\Mysim\phi) = 0$.
\end{lemma}

\begin{proof}
    The proof is by induction on the length of $\phi$.
\end{proof}

As highlighted earlier, $\mathsf{K}_3$ and $\mathsf{LP}$ are negation dual to each other, and $\mathsf{ST}$ and $\mathsf{TS}$ are negation self-dual. We prove here that in a similar way $\mathsf{K}_3$ and $\mathsf{LP}$ are operationally dual, and $\mathsf{ST}$ and $\mathsf{TS}$ operationally self-dual.

\begin{proposition}\label{thm:oduaK3LP}
    $\models_{\mathsf{K}_3} \Gamma \Rightarrow \Delta$ if and only if $\models_{\mathsf{LP}} \Mysim(\Delta) \Rightarrow \Mysim(\Gamma)$ and $\models_{\mathsf{LP}} \Gamma \Rightarrow \Delta$ if and only if $\models_{\mathsf{K}_3} \Mysim(\Delta) \Rightarrow \Mysim(\Gamma)$.
\end{proposition}

\begin{proof}
    Assume $\models_{\mathsf{K}_3} \Gamma \Rightarrow \Delta$ and $\val(\Mysim \delta) \neq 0$ for all $\delta \in \Delta$. By the preceding lemma, $v(\delta) \neq 1$ for all $\delta \in \Delta$, so $v(\gamma) \neq 1$ for some $\gamma \in \Gamma$. Again, by the preceding lemma, $\val(\Mysim \gamma) \neq 0$ for some $\gamma \in \Gamma$, and therefore $\models_{\mathsf{LP}} \Mysim(\Delta) \Rightarrow \Mysim(\Gamma)$. Assume now $\models_{\mathsf{LP}} \Mysim(\Delta) \Rightarrow \Mysim(\Gamma)$ and $v(\gamma) =1$ for all $\gamma \in \Gamma$. By the preceding lemma, $\val(\Mysim\gamma) = 0$ for all $\gamma \in \Gamma$, so $\val(\Mysim\delta) = 0$ for some $\delta \in \Delta$. Again, by the preceding lemma, $v( \delta) = 1$ for some $\delta \in \Delta$, and therefore $\models_{\mathsf{K}_3} \Gamma \Rightarrow \Delta$. The second equivalence can be proved symmetrically.
\end{proof}

\begin{proposition}\label{thm:oduaSTTS}
    $\models_{\mathsf{ST}} \Gamma \Rightarrow \Delta$ if and only if $\models_{\mathsf{ST}} \Mysim(\Delta) \Rightarrow \Mysim(\Gamma)$ and $\models_{\mathsf{TS}} \Gamma \Rightarrow \Delta$ if and only if $\models_{\mathsf{TS}} \Mysim(\Delta) \Rightarrow \Mysim(\Gamma)$.
\end{proposition}

\begin{proof}
   Assume $\models_{\mathsf{ST}} \Gamma \Rightarrow \Delta$ and $\val(\Mysim \delta) =1$ for all $\delta \in \Delta$. By the preceding lemma, $v(\delta) = 0$ for all $\delta \in \Delta$, so $v(\gamma) \neq 1$ for some $\gamma \in \Gamma$. Again, by the preceding lemma, $\val(\Mysim \gamma) \neq 0$ for some $\gamma \in \Gamma$, and therefore $\models_{\mathsf{ST}} \Mysim(\Delta) \Rightarrow \Mysim(\Gamma)$. Assume now $\models_{\mathsf{ST}} \Mysim(\Delta) \Rightarrow \Mysim(\Gamma)$ and $v(\gamma) =1$ for all $\gamma \in \Gamma$. By the preceding lemma, $\val(\Mysim\gamma) =0$ for all $\gamma \in \Gamma$, so $\val(\Mysim\delta) \neq 1$ for some $\delta \in \Delta$. Again, by the preceding lemma, $v(\delta) \neq 0$ for some $\delta \in \Delta$, and therefore $\models_{\mathsf{ST}} \Gamma \Rightarrow \Delta$. The case of $\mathsf{TS}$ can be proved similarly.
\end{proof}

\begin{defn}[Operational duality for logics]
   For any two logics $\mathsf{L}_1$ and $\mathsf{L}_2$ based on the language of Definition \ref{language2}, we will say that $\mathsf{L}_1$ is operationally dual to $\mathsf{L}_2$ when $\models_{\mathsf{L}_1} \Gamma \Rightarrow \Delta$ if and only if $\models_{\mathsf{L}_2} \Mysim(\Delta) \Rightarrow \Mysim(\Gamma)$ for all $\Gamma, \Delta \subseteq \mathcal{L}$.  
\end{defn}

Given this terminology, $\mathsf{K}_3$ and $\mathsf{LP}$ are therefore operationally dual, while $\mathsf{ST}$ and $\mathsf{TS}$ are both self-dual.
 

It remains to show that such notion of duality licenses an extensional interdefinability between a logic and its operational dual. This is proven in the next lemma.

\begin{lemma}\label{lem:ext}
Given two logics $\mathsf{L}_1$ and $\mathsf{L}_2$ defined on the language of Definition \ref{language2}, $\mathsf{L}_1$ is operationally dual to $\mathsf{L}_2$ if and only if $\mathbb{L}_1^{+}=\Mysim(\mathbb{L}_2^{+})^{-1}$.
\end{lemma}

\begin{proof}
    Assume that a logic $\mathsf{L}_1$ is operationally dual to a logic $\mathsf{L}_2$. Then $\Gamma \Rightarrow \Delta \in \mathbb{L}^{+}_1$ if and only if $\Mysim(\Delta) \Rightarrow \Mysim(\Gamma) \in \mathbb{L}^{+}_2$. Which is equivalent to $\Mysim(\Gamma) \Rightarrow \Mysim(\Delta) \in (\mathbb{L}_2^{+})^{-1}$, and in turn to $\Mysim\Mysim(\Gamma) \Rightarrow \Mysim\Mysim(\Delta) \in \Mysim(\mathbb{L}_2^{+})^{-1}$. Given that $\Mysim$ is involutive by Lemma \ref{lem:inv}, the previous statement amounts to $\Gamma \Rightarrow \Delta \in \Mysim(\mathbb{L}_2^{+})^{-1}$.

    Assume now that $\mathbb{L}_1^{+}=\Mysim(\mathbb{L}_2^{+})^{-1}$. We have that $\Gamma \Rightarrow \Delta \in \mathbb{L}_1^{+}$ if and only if $\Gamma \Rightarrow \Delta \in \Mysim(\mathbb{L}_2^{+})^{-1}$. But this is equivalent to  $\Mysim\Mysim\Gamma \Rightarrow \Mysim\Mysim\Delta \in \Mysim(\mathbb{L}_2^{+})^{-1}$ since $\Mysim$ is involutive. Which is to say that $\Mysim\Gamma \Rightarrow \Mysim\Delta \in (\mathbb{L}_2^{+})^{-1}$, and in turn that $\Mysim\Delta \Rightarrow \Mysim\Gamma \in \mathbb{L}_2^{+}$.
\end{proof}

With this lemma at hand, we are in a position to characterize extensionally each of the logics discussed in terms of their operational dual.

\begin{theorem}\label{thm:ext}
\[\arraycolsep=10pt\def\arraystretch{2}
\begin{array}{cc}    
    \mathbb{K}_3^{+} = \Mysim (\mathbb{LP}^{+})^{-1} & \mathbb{LP}^{+} = \Mysim (\mathbb{K}_3^{+})^{-1}\\
    
    \mathbb{ST}^{+} = \Mysim (\mathbb{ST}^{+})^{-1} & \mathbb{TS}^{+} = \Mysim(\mathbb{TS}^{+})^{-1}
   \end{array}\]
\end{theorem}
\begin{proof}\leavevmode
This follows directly from Lemma \ref{lem:ext}, Proposition \ref{thm:oduaK3LP} and Proposition \ref{thm:oduaSTTS}.
\end{proof}

The composition $\Mysim \circ ^{-1}$ is therefore a function mapping a set of valid inferences of a logic to the set of valid inferences of its operational dual. $\mathbb{LP}^{+}$ is mapped to $\mathbb{K}_3^{+}$, $\mathbb{K}_3^{+}$ to 
$\mathbb{LP}^{+}$, while both $\mathbb{ST}^{+}$ and $\mathbb{TS}^{+}$ are fixed points of the composition. 

\subsection{Structural Duality}

In the second half of this section, we combine our proposed notion of operational duality with the notion of structural duality to show that $\mathsf{ST}$ and $\mathsf{TS}$ are each other's dual under this combined notion, in the same way as $\mathsf{K}_3$ and $\mathsf{LP}$ are each other's dual. In addition, the results of sections \ref{sec:st} and \ref{sec:ts} will allow for an extensional characterization of the sets of $\mathsf{ST}$- and $\mathsf{TS}$-valid inferences solely from the set of $\mathsf{K}_3$- or $\mathsf{LP}$-valid inferences. Along the way, we will prove a similar duality result for the sets of $\mathsf{K}_3$ and $\mathsf{LP}$ \textit{antivalid} inferences, and see that the operations of relative product and relative sum can be used to characterize the sets of $\mathsf{ST}$ and $\mathsf{TS}$ antivalid inferences out of the sets of $\mathsf{LP}$ and $\mathsf{K}_3$ antivalid inferences.

The notion of \textit{antivalidity} used throughout this section was first introduced by \cite{cobreros2021cant} to account for the notion of structural duality. It is defined as follows:

\begin{defn}[$\mathsf{L}$-antisatisfaction, $\mathsf{L}$-antivalidity]
Let $\mathsf{L}$ be a logic based on a consequence relation defined on the standard $D_1$ for the premises and $D_2$ for the conclusions. 
\begin{itemize}
\item An SK-valuation $v$ \textit{antisatisfies} an inference $\Gamma \Rightarrow \Delta$ in $\mathsf{L}$ (symbolized by $v \nmodels_{\mathsf{L}} \Gamma \Rightarrow \Delta$) if and only if $v(\gamma) \not\in D_1$ for all $\gamma \in \Gamma$ only if $v(\delta) \not\in D_2$ for some $\delta \in \Delta$. 
\item An inference $\Gamma \Rightarrow \Delta$ is \textit{antivalid} in $\mathsf{L}$ (symbolized by $\nmodels_{\mathsf{L}} \Gamma \Rightarrow \Delta)$ if and only if $v \nmodels_{\mathsf{L}} \Gamma \Rightarrow \Delta$ for all SK-valuations $v$.
\end{itemize}
\end{defn}

For instance, $\nmodels_{\mathsf{LP}} p \Rightarrow p \wedge q$, since for all $v$, if $v(p) \not\in \lbrace \sfrac{1}{2}, 1\rbrace$, then $v(p \wedge q) \not\in \lbrace \sfrac{1}{2}, 1\rbrace$; while $\not\nmodels_{\mathsf{ST}} p \Rightarrow p \wedge q$, because there is $v$ such that $v(p)=v(q)=v(p \wedge q) =\sfrac{1}{2}$, and $v(p) \neq 1$ but $v(p \wedge q) \in \lbrace \sfrac{1}{2}, 1\rbrace$.

\textcite{cobreros2021cant} say that a logic is structurally dual to another whenever, when an argument is valid in the first logic, the converse (or inverse) argument is antivalid in the second logic. Formally:


\begin{defn}[Structural duality]
    A logic $\mathsf{L}_1$ is structurally dual to a logic $\mathsf{L}_2$ when
   \[\begin{array}{ccc}
        \models_{\mathsf{L}_1} \Gamma \Rightarrow \Delta & \textup{iff} & \nmodels_{\mathsf{L}_2} \Delta \Rightarrow \Gamma
    \end{array}\]
\end{defn}

As noted before, structural duality implicitly involves a form of  negation in the metalanguage, more precisely, a form of metalinguistic contraposition. A notion of validity is expressed through a universally closed metalinguistic conditional statement connecting the satisfaction of an inference's premises to the satisfaction of its conclusions. In turn, antivalidity is obtained by contraposing this metalinguistic conditional while keeping fixed the sets of designated values. Structural duality between two logics $\mathsf{L}_1$ and $\mathsf{L}_2$ requires therefore an equivalence between two universally closed conditional statements: the first one connecting the $\mathsf{L}_1$-satisfaction of an inference's premises to the $\mathsf{L}_1$-satisfaction of its conclusions, and the second one connecting the non-$\mathsf{L}_2$-satisfaction of its conclusions to the non-$\mathsf{L}_2$-satisfaction of its premises.  

The following proposition is proven by Cobreros and associates:

\begin{proposition}[Structural Duals][\cite{cobreros2021cant}]\label{prop:dua}
\begin{align*}
    \models_{\mathsf{LP}} \Gamma \Rightarrow \Delta \ \textup{if and only if} \ \nmodels_{\mathsf{LP}} \Delta \Rightarrow \Gamma\\
    \models_{\mathsf{K}_3} \Gamma \Rightarrow \Delta \ \textup{if and only if} \ \nmodels_{\mathsf{K}_3} \Delta \Rightarrow \Gamma\\
    \models_{\mathsf{ST}} \Gamma \Rightarrow \Delta \ \textup{if and only if} \ \nmodels_{\mathsf{TS}} \Delta \Rightarrow \Gamma\\
    \models_{\mathsf{TS}} \Gamma \Rightarrow \Delta \ \textup{if and only if} \ \nmodels_{\mathsf{ST}} \Delta \Rightarrow \Gamma
\end{align*}
\end{proposition}

\begin{proof}
    Note that if $\mathsf{XY}$ and $\mathsf{YX}$ are two logics defined on the same scheme for the connectives and $\mathsf{XY}$ is a logic based on a mixed consequence relation defined on the standard $X$ for the premises and $Y$ for the conclusions (and conversely for $\mathsf{YX}$), then $\models_{\mathsf{XY}} \Gamma \Rightarrow \Delta$ if and only if $\nmodels_{\mathsf{YX}} \Delta \Rightarrow \Gamma$.
\end{proof}

The set of antivalidities of a logic is defined in a similar way as the set of validities of a logic:

\begin{defn}[Set of antivalidities]\label{def:ant}
\begin{align*}
\mathbb{L}^{-} &= \lbrace \Gamma \Rightarrow \Delta  : \ \nmodels_{\mathsf{L}} \Gamma \Rightarrow \Delta \rbrace
\end{align*}
\end{defn}

Proposition \ref{prop:dua} can thus be rephrased in terms of inverse: the sets of $\mathsf{K}_3$- and $\mathsf{LP}$-antivalid inferences are respectively the inverses of the sets of $\mathsf{K}_3$- and $\mathsf{LP}$-valid inferences; while the sets of $\mathsf{ST}$- and $\mathsf{TS}-$antivalid inferences are respectively the inverses of the sets of $\mathsf{TS}$- and $\mathsf{ST}$-valid inferences.

\begin{corollary}[Inverse of $\mathbb{K}_3^{+}$, $\mathbb{LP}^{+}$, $\mathbb{ST}^{+}$ and $\mathbb{TS}^{+}$]\label{cor:inv}
    \begin{align*}
    \mathbb{LP}^{-} &= (\mathbb{LP}^{+})^{-1}\\
\mathbb{K}_3^{-} &= (\mathbb{K}_3^{+})^{-1}\\
\mathbb{TS}^{-} &= (\mathbb{ST}^{+})^{-1}\\
\mathbb{ST}^{-} &= (\mathbb{TS}^{+})^{-1}
\end{align*}
\end{corollary}

\begin{proof}
    This follows from Proposition \ref{prop:dua}.
\end{proof}

The results of Theorem \ref{thm:ext} can now be restated in light of these equalities.

\begin{theorem}\label{thm:extbis}
\[\arraycolsep=10pt\def\arraystretch{2}
\begin{array}{cc}    
    \mathbb{K}_3^{+} = \Mysim \mathbb{LP}^{-} & \mathbb{LP}^{+} = \Mysim \mathbb{K}_3^{-}\\
    
    \mathbb{ST}^{+} = \Mysim \mathbb{TS}^{-} & \mathbb{TS}^{+} = \Mysim\mathbb{ST}^{-}
   \end{array}\]
\end{theorem}

\begin{proof}
    From Corollary \ref{cor:inv} and Theorem \ref{thm:ext}.
\end{proof}

A dual result can be given by interchanging the validities and the antivalidities:

\begin{theorem}\label{cor:ext-}
\[\arraycolsep=10pt\def\arraystretch{2}
\begin{array}{cc}    
    \mathbb{K}_3^{-} = \Mysim \mathbb{LP}^{+} & \mathbb{LP}^{-} = \Mysim \mathbb{K}_3^{+}\\
    
    \mathbb{ST}^{-} = \Mysim \mathbb{TS}^{+} & \mathbb{TS}^{-} = \Mysim\mathbb{ST}^{+}
   \end{array}\]
\end{theorem}

\begin{proof}
    We only prove the first equivalence, the proof of the others being similar. By Corollary \ref{cor:inv}, $\mathbb{K}_3^{-} =(\mathbb{K}_3^{+})^{-1}$, and by the previous theorem $(\mathbb{K}_3^{+})^{-1} = (\Mysim\mathbb{LP}^{-})^{-1} = (\Mysim((\mathbb{LP}^{+})^{-1}))^{-1}$. Finally, by Lemma $\ref{lem:asso}$, $\Mysim(((\mathbb{LP}^{+})^{-1})^{-1}) = \Mysim(\mathbb{LP}^{+})$.
\end{proof}

Although neither the revised notion of operational duality nor the notion of structural duality makes $\mathsf{K}_3$ and $\mathsf{LP}$ dual simultaneously with $\mathsf{ST}$ and $\mathsf{TS}$, the combination of the two leads to a unified extensional definition of $\mathsf{K}_3$ in terms of $\mathsf{LP}$ and of $\mathsf{ST}$ in terms of $\mathsf{TS}$ (and conversely). In the case of $\mathsf{K}_3$ and $\mathsf{LP}$, this combined notion of duality collapses with operational duality, since both are structurally self-dual; in the case of $\mathsf{ST}$ and $\mathsf{TS}$, the notion collapses with structural duality since both are operationally self-dual.

\subsection{Putting it all together}

We now turn to the characterization of the sets of $\mathsf{ST}$ and $\mathsf{TS}$ antivalid inferences out of the sets of $\mathsf{LP}$ and $\mathsf{K}_3$ antivalid inferences. It is known that the inverse of the relational product of two relations $R$ and $S$ is equal to the relational product of the inverse of $S$ and the inverse of $R$, and similarly for the relational sum \parencite[]{Maddux2006}. The next fact attests that this remains true for the modified versions of relative product and relative sum introduced in Definition \ref{DefRelP} and Definition \ref{DefRelS}.

\begin{fact}[Inverse product and inverse sum]\label{fct:inv}
\begin{align*}
(R \hspace{0.1em} \stackunder{$\mid$}{$\ms V$} \hspace{0.1em} S)^{-1} &= S^{-1} \hspace{0.1em} \stackunder{$\mid$}{$\ms V$} \hspace{0.1em} R^{-1}\\
(R \hspace{0.2em} \stackunder{$\dag$}{$\ms V$} \hspace{0.2em} S)^{-1} &= S^{-1} \hspace{0.1em} \stackunder{$\dag$}{$\ms V$} \hspace{0.1em} R^{-1}
\end{align*}
\end{fact}

\begin{proof}
   \noindent\begin{paracol}{2}
\begin{align*}
&\langle x, z\rangle \in (R \hspace{0.1em} \stackunder{$\mid$}{$\ms V$} \hspace{0.1em} S)^{-1}\\
\text{iff} \ &\langle z, x \rangle \in  R \hspace{0.1em} \stackunder{$\mid$}{$\ms V$} \hspace{0.1em} S\\
\text{iff} \ &\textup{there is} \ y \in V \ \textup{s.t.} \ \langle z, y\rangle \in R\\
&\text{and} \ \langle y, x\rangle \in S \\
\text{iff} \ &\textup{there is} \ y \in V \ \textup{s.t.} \ \langle y, z\rangle \in R^{-1}\\
&\text{and} \ \langle x, y\rangle \in S^{-1} \\
\text{iff} \ &\langle x, z\rangle \in S^{-1} \hspace{0.1em} \stackunder{$\mid$}{$\ms V$} \hspace{0.1em} R^{-1}.
\end{align*}
\switchcolumn[1]
\begin{align*}
&\langle x, z\rangle \in (R \hspace{0.2em} \stackunder{$\dag$}{$\ms V$} \hspace{0.2em} S)^{-1}\\
\text{iff} \ &\langle z, x \rangle \in  R \hspace{0.2em} \stackunder{$\dag$}{$\ms V$} \hspace{0.2em} S\\
\text{iff} \ &\textup{for all} \ y \in V,\langle z, y\rangle \in R\\
&\text{or} \ \langle y, x\rangle \in S \\
\text{iff} \ &\textup{for all} \ y \in V, \langle y, z\rangle \in R^{-1}\\
&\text{or} \ \langle x, y\rangle \in S^{-1} \\
\text{iff} \ &\langle x, z\rangle \in S^{-1} \hspace{0.2em} \stackunder{$\dag$}{$\ms V$} \hspace{0.2em} R^{-1}.
\end{align*}
\end{paracol}
\end{proof}

We are now in a position to identify the set of $\mathsf{ST}$-antivalid inferences with the relative sum of the sets of $\mathsf{K}_3$-antivalid inferences and $\mathsf{LP}$-antivalid inferences.

\begin{theorem}\label{thm:st-}
$\mathbb{ST}^{-} = \mathbb{K}_3^{-} \ \dag \ \mathbb{LP}^{-}$
\end{theorem}

\begin{proof}
By Corollary \ref{cor:inv}, $\mathbb{ST}^{-} = (\mathbb{TS}^{+})^{-1}$, moreover $\mathbb{TS}^{+} = \mathbb{LP}^{+} \ \dag \ \mathbb{K}_3^{+}$ by Theorem \ref{thm:ts}, so $(\mathbb{TS}^{+})^{-1} = (\mathbb{LP}^{+} \ \dag \ \mathbb{K}_3^{+})^{-1}$. Finally, by Fact \ref{fct:inv} and Corollary \ref{cor:inv} again, $(\mathbb{LP}^{+} \ \dag \ \mathbb{K}_3^{+})^{-1} = (\mathbb{K}_3^{+})^{-1} \ \dag \ (\mathbb{LP}^{+})^{-1} = \mathbb{K}_3^{-} \ \dag \ \mathbb{LP}^{-}$.
\end{proof}

The set of $\mathsf{TS}$-antivalid inferences can in turn be characterized as the relative product of the set of $\mathsf{LP}$-antivalid inferences and the set of $\mathsf{K}_3$-antivalid inferences.

\begin{theorem}\label{thm:ts-}
$\mathbb{TS}^{-} = \mathbb{LP}^{-} \mid \mathbb{K}_3^{-}$
\end{theorem}

\begin{proof}
By Corollary \ref{cor:inv}, $\mathbb{TS}^{-} = (\mathbb{ST}^{+})^{-1}$, moreover $\mathbb{ST}^{+} = \mathbb{K}_3^{+} \mid \mathbb{LP}^{+}$ by Theorem \ref{thm:st}, so $(\mathbb{ST}^{+})^{-1} = (\mathbb{K}_3^{+} \mid \mathbb{LP}^{+})^{-1}$. Finally, by Fact \ref{fct:inv} and Corollary \ref{cor:inv} again, $(\mathbb{K}_3^{+} \mid \mathbb{LP}^{+})^{-1} = (\mathbb{LP}^{+})^{-1} \mid (\mathbb{K}_3^{+})^{-1} = \mathbb{LP}^{-} \mid \mathbb{K}_3^{-}$.
\end{proof}

Combining Theorem \ref{thm:st} and Theorem \ref{thm:ts} with Theorem \ref{thm:extbis}, and Theorem \ref{thm:st-} and \ref{thm:ts-} with Corollary \ref{cor:ext-} leads to the following proposition, which shows that $\mathsf{ST}$ and $\mathsf{TS}$ can be extensionally characterized by either $\mathsf{LP}$ or $\mathsf{K}_3$.

\begin{theorem}\label{cor:int}
\[\arraycolsep=10pt\def\arraystretch{2}
\begin{array}{cc} 
\mathbb{ST}^{+}= \mathbb{K}_3^{+} \mid \Mysim\mathbb{K}_3^{-} = \Mysim\mathbb{LP}^{-} \mid \mathbb{LP}^{+} &
 \mathbb{TS}^{+}= \Mysim\mathbb{K}_3^{-} \ \dag \ \mathbb{K}_3^{+} = \mathbb{LP}^{+} \ \dag \ \Mysim\mathbb{LP}^{-}\\
 \mathbb{ST}^{-}= \mathbb{K}_3^{-} \ \dag \ \Mysim\mathbb{K}_3^{+} = \Mysim\mathbb{LP}^{+} \ \dag \ \mathbb{LP}^{-} &
 \mathbb{TS}^{-}= \Mysim\mathbb{K}_3^{+} \mid \mathbb{K}_3^{-} = \mathbb{LP}^{-} \mid \Mysim\mathbb{LP}^{+}
\end{array}\]
\end{theorem}

\begin{proof}
    This follows directly from Theorem \ref{thm:st}, Theorem \ref{thm:ts}, Theorem \ref{thm:st-}, Theorem \ref{thm:ts-}, Theorem \ref{thm:extbis} and Corollary \ref{cor:ext-}.
\end{proof}

Those identities lend support to the view that $\mathsf{ST}$ isn't more fundamentally tied to $\mathsf{LP}$ than it is to $\mathsf{K}_3$ (see \cite{cobreros2020inferences}): either system can be used as a primitive in order to define either $\mathsf{ST}$ or $\mathsf{TS}$ at the inferential level.

\section{Comparisons}\label{sec:discussion}

We have shown that $\mathsf{ST}$ and $\mathsf{TS}$ can be extensionally defined from either $\mathsf{K}_3$ or $\mathsf{LP}$. In this section, we comment on the significance of these findings by comparing them to two sets of results found in the literature. We start by discussing the link between the characterization proposed of $\mathsf{ST}$ and a genereralization of Craig's interpolation theorem proven by \cite{milne2016refinement} for classical logic. Then, we return to the significance of the duality results presented in the previous section, in particular in comparison to those of \cite{cobreros2021cant} and \cite{DaRéetal2020}.

\subsection{Interpolation}
Given an $\mathsf{ST}$-valid inference $\Gamma\Rightarrow \Delta$, Theorem \ref{thm:st} gives us a constructive method of finding a formula $\phi$ such that $\Gamma \Rightarrow \phi$ is $\mathsf{K}_3$-valid, and $\phi \Rightarrow \Delta$ is $\mathsf{LP}$-valid, namely the $\mathsf{K}_3$-disjunctive normal form of $\bigwedge \Gamma$. We called such a formula a connecting formula, but in the example we produced, it has the features of an interpolant in the sense of Craig's theorem for propositional logic. Craig's interpolation theorem for classical logic states that when $\phi \Rightarrow \psi$ is classically valid, and provided $\phi$ is not a contradiction and $\psi$ not a tautology, there is an interpolant formula $\chi$, consisting of only atoms common to $\phi$ and $\psi$, such that $ \phi \Rightarrow \chi$ and $\chi \Rightarrow \psi$ are classically valid. Hence, it is natural to wonder about the connection between our result and Craig's interpolation theorem.

As it turns out, \cite{milne2016refinement} proved a refinement of Craig's theorem, which is that when $\phi$ classically entails $\psi$ and neither $\phi$ is contradiction nor $\psi$ a tautology, then there is an interpolant formula $\chi$ such that $\phi\Rightarrow \chi$ is $\mathsf{K}_3$-valid, and $\chi\Rightarrow \psi$ is $\mathsf{LP}$-valid (see also \cite{Prenosil2017} for a statement of other interpolation properties for three-valued logics in the vicinity of $\mathsf{LP}$ and $\mathsf{K}_3$). An examination of Milne's proof shows that he used the same method we used {in the first half of the proof of Theorem \ref{thm:st}}, except that he restricts it to atoms common to the antecedent and succedent of an argument. In the example we gave of $p \vee (q \wedge \neg q) \Rightarrow p \wedge (q \vee \neg q) $, Milne's method would produce as interpolating formula the same as ours, namely $(p \wedge q) \vee (p \wedge \neg q) \wedge p$. But in the variant $p \vee (q \wedge \neg q) \Rightarrow p \wedge (r \vee \neg r)$, Milne's method would produce $p$ as the interpolant given Craig's constraint on atom-sharing. We do not have that constraint and our method would still produce $(p \wedge q) \vee (p \wedge \neg q) \wedge p$ as the connecting formula, since it is based solely on producing a $\mathsf{K}_3$-DNF.

 Milne's result is not concerned with $\mathsf{ST}$, it is also limited to the case of inferences with a {satisfiable} single premise and a {falsifiable} single conclusion, and it is proven for the language of propositional logic without constants. However, it bears a direct connection to our characterization of $\mathsf{ST}$, since it can be used to show that classical logic is the product of $\mathsf{K}_3$ and $\mathsf{LP}$. This raises a more general question, namely whether we may strengthen the characterization of $\mathsf{ST}$ in terms of relational product by making systematic reference to the interpolation property. In the language involving constants, we know that $\top \Rightarrow \lambda$ is $\mathsf{ST}$-valid, and in this case the only connecting formula is $\top$, since $\top \Rightarrow \top$ is $\mathsf{K}_3$-valid, and $\top \Rightarrow \lambda$ is $\mathsf{LP}$-valid, so the usual notion of interpolant does not obtain.
 
 But the question remains of more general interest. In particular, as discussed in \cite{dare2023three}, and as originally proved by \cite{szmuc2021meaningless}, the logic $\mathsf{ST}$ can also be obtained in terms of mixed consequence on the basis of the Weak Kleene logic $\mathsf{WK}$ and its dual system Paraconsistent Weak Kleene $\mathsf{PWK}$: both systems rest on the Weak Kleene scheme, unlike $\mathsf{K}_3$ and $\mathsf{LP}$. So does $\mathsf{ST}$ equal the relational product of $\mathsf{WK}$ and $\mathsf{PWK}$, and if so, is the connecting formula an interpolant? More generally, \cite{dare2023three} show that $\mathsf{ST}$ can be characterized in terms of mixed consequence over even more schemes intermediate between the Weak Kleene and the Strong Kleene scheme: for those logics, does the product characterization still hold, and if so, is the connecting formula always an interpolant in the sense of Craig? We reserve the answer to these questions for further work.

\subsection{Duality}
Several duality results have been offered in the literature on Strong Kleene logics. As reviewed above, \textcite{Cobrerosetal2012} have shown that $\mathsf{K}_3$ and $\mathsf{LP}$ are negation dual, and $\mathsf{ST}$ and $\mathsf{TS}$ negation self-dual; \textcite{cobreros2021cant} that $\mathsf{ST}$ and $\mathsf{TS}$ are structurally dual and $\mathsf{K}_3$ and $\mathsf{LP}$ structurally self-dual; while \textcite{DaRéetal2020} that $\mathsf{ST}$ and $\mathsf{TS}$ are metainferentially dual.\footnote{A metainference is an inference between inferences. \textcite{DaRéetal2020} show, with the help of a notion of antisatisfaction different from the one used here, that every metainference can be dualized. They prove that for every $\mathsf{ST}$-valid metainference, its dual is $\mathsf{TS}$-valid (and conversely).} Yet, none of these notions of duality makes $\mathsf{ST}$ the dual of $\mathsf{TS}$ in the same sense in which $\mathsf{K}_3$ is the dual of $\mathsf{LP}$.

The notion of operational duality put forward in this paper differs from each of these notions. The difference with structural duality and metainferential duality is readily noticeable, since none of them involves connectives and are thus not operational. The difference with negation duality is less obvious, both notions involve operators and make $\mathsf{K}_3$ dual to $\mathsf{LP}$ and $\mathsf{ST}$ and $\mathsf{TS}$ self-dual. But as shown earlier, no function can be straightforwardly defined from negation duality with the intent of characterizing the set of $\mathsf{K}_3$-valid inferences from the set of $\mathsf{LP}$-valid inferences, contrary to operational duality, where such a function was easily specified. Moreover, negation duality only involves a single operator, negation, whereas operational duality involves all the operators of the language. Given a formula, all the operators occurring in it are interchanged by the function $\Mysim$ with their dual: $\top$ with $\bot$, $\lambda$ with $\lambda$, $\neg$ with $\neg$, and $\wedge$ with $\lor$. To see why for all the logics discussed above this duality between operators obtains, it is enough to notice that if $1$ and $0$ are renamed $0$ and $1$, then the truth-conditions of $\bot$ and $\top$, and $\wedge$ and $\lor$ are swapped, while the truth-conditions of $\lambda$ and $\neg$ are left unchanged.\footnote{See \cite[p.~23-24]{Kleene1967-KLEML} for an illustration of this duality between operators that can easily be adapted to the logics discussed here.} Despite these differences, the two notions are connected through the following fact:

\begin{fact}
For any truth-functional logic $\mathsf{L}$ based on the language of Definition \ref{language2} and equipped with an involutive negation satisfying the De Morgan laws:
\[\arraycolsep=10pt\def\arraystretch{2}
\begin{array}{ccc}
    \models_{\mathsf{L}} \Mysim (\Gamma) \Rightarrow \Mysim(\Delta) & \textup{if and only if} & \models_{\mathsf{L}} \neg(\Gamma) \Rightarrow \neg(\Delta).
 \end{array}\]
 \end{fact}

\begin{proof}
    Let $\sigma$ be a substitution of every atom by its negation. By a simple but tedious induction one can show that for all $v$, $v(\sigma(\Mysim \phi)) = v(\neg \phi)$ and $v(\sigma(\neg \phi)) = v(\Mysim \phi)$. Hence, if $\models_{\mathsf{L}} \Mysim (\Gamma) \Rightarrow \Mysim(\Delta)$, then $\models_{\mathsf{L}} \sigma(\Mysim (\Gamma)) \Rightarrow \sigma(\Mysim(\Delta))$, and therefore $\models_{\mathsf{L}} \neg(\Gamma) \Rightarrow \neg(\Delta)$. The other direction is similar. 
\end{proof}

This explains why both negation duality and operational duality make $\mathsf{K}_3$ dual to $\mathsf{LP}$ and $\mathsf{ST}$ and $\mathsf{TS}$ self-dual. In the context of truth-functional logics based on the language of Definition \ref{language2} and equipped with an involutive negation satisfying the De Morgan law, the two notions are equivalent.


Still, our notion of operational duality, like the others mentioned, does not by itself offer an account that simultaneously makes $\mathsf{ST}$ and $\mathsf{TS}$ dual, and $\mathsf{K}_3$ and $\mathsf{LP}$ dual. It is only when combined with structural duality, that such a unified account can be obtained. Theorems \ref{thm:extbis} and \ref{cor:ext-} thus provide a uniform characterization of $\mathsf{K}_3$ in terms of $\mathsf{LP}$ and of $\mathsf{ST}$ in terms of $\mathsf{TS}$. Given one of these logics, one can obtain the set of (antivalid) valid inferences of its dual by first taking the set of (valid) antivalid inferences of this logic and then applying the mapping $\Mysim$ to this set.  $\mathsf{ST}$ is dual to $\mathsf{TS}$ under this combined notion inasmuch as $\mathsf{ST}$ is operationally self-dual but structurally dual to $\mathsf{TS}$; $\mathsf{K}_3$ is dual to $\mathsf{LP}$ under this same notion inasmuch as $\mathsf{K}_3$ is structurally self-dual but operationally dual to $\mathsf{LP}$.



As a concluding remark, let us note that since its inception, $\mathsf{ST}$ has been presented as being of type $\mathsf{K}_3$ on the side of premises and of type $\mathsf{LP}$ on the side of conclusions (viz. \cite{cobreros2015vagueness,cobreros2020inferences}), based on the very definition of $\mathsf{ST}$ in terms of mixed standards of truth. For example, \cite{zardini2022final} calls $\mathsf{ST}$ ``$\mathsf{K}_3\mathsf{LP}$'', precisely on account of this hybrid behavior. One might therefore be tempted to call $\mathsf{TS}$ ``$\mathsf{LP}\mathsf{K}_3$'' on the same grounds. The results of this paper may be used to vindicate those denominations. But what matters is to see that the operation that connects  $\mathsf{K}_3$ and $\mathsf{LP}$ extensionally is not the same depending on the case. In particular, as proven in section \ref{sec:converse}, $\mathbb{TS}$ is not the relational product of $\mathbb{LP}$ and $\mathbb{K}_3$, and conversely $\mathbb{ST}$ is not the relational sum of $\mathbb{K}_3$ and $\mathbb{LP}$. What is more, the results proven in Theorem \ref{cor:int} from the duality of $\mathsf{K}_3$ and $\mathsf{LP}$ might equally license  alternative denominations of $\mathsf{ST}$ and $\mathsf{TS}$ as ``$\mathsf{K}_3\mathsf{K}_3$'' or ``$\mathsf{LP}\mathsf{LP}$'', which blurs the distinction between the two logics. Pending a conclusive argument favoring the characterization of a logic in terms of validity rather than antivalidity, we do not see why $\mathsf{ST}$ and $\mathsf{TS}$ should deserve the denominations ``$\mathsf{K}_3\mathsf{LP}$'' and ``$\mathsf{LP}\mathsf{K}_3$'', more than ``$\mathsf{K}_3\mathsf{K}_3$'' or ``$\mathsf{LP}\mathsf{LP}$''.




\section{Conclusion}


In this paper we have proved that $\mathsf{ST}$ and $\mathsf{TS}$ each have a natural set-theoretic characterization based on the inferences of the dual logics $\mathsf{K}_3$ and $\mathsf{LP}$: $\mathsf{ST}$ is the relational product of the inferences of $\mathsf{K}_3$ and $\mathsf{LP}$, while $\mathsf{TS}$ is their relational sum. 

We reckon that those results are of importance philosophically and proof-theoretically. Philosophically, because whereas $\mathsf{ST}$ and $\mathsf{TS}$ are semantically defined using a notion of mixed consequence, the results given here provide a characterization of each logic directly in terms of the inferences of $\mathsf{K}_3$ and $\mathsf{LP}$, without explicitly invoking shifting standards of truth between premises and conclusions. Proof-theoretically, because $\mathsf{ST}$ and $\mathsf{TS}$ are obtained directly via some operation on the logics $\mathsf{K}_3$ and $\mathsf{LP}$. For $\mathsf{ST}$, in particular, one can determine whether $\Gamma$ $\mathsf{ST}$-entails $\Delta$ by constructing the $\mathsf{K}_3$-disjunctive normal form $\Gamma'$ of $\Gamma$ and by checking if $\Gamma'$ $\mathsf{LP}$-entails $\Delta$.

The present characterization is of interest more generally, in particular when considering the metainferential hierarchy discussed by \cite{Barrioetal2020}. Barrio and associates have shown that $\mathsf{ST}$ and $\mathsf{TS}$ are both only the first level of a more extended hierarchy of systems of mixed consequence. The system $\mathsf{TS}/\mathsf{ST}$, for example, is defined by letting $\Gamma \Rightarrow \Delta$ be in $\mathsf{TS}/\mathsf{ST}$ provided every SK-valuation $v$ either fails to $\mathsf{TS}$-satisfy $\gamma$ for some $A$ in $\Gamma$ or $\mathsf{ST}$-satisfies $B$ for some $B$ in $\Delta$ (with $A$, $B$ inferences). This raises the question of whether $\mathsf{TS}/\mathsf{ST}$ can also be defined as the relational product of $\mathsf{TS}$ and $\mathsf{ST}$, and whether dual results can be obtained for the $\mathsf{TS}$-hierarchy. As it turns out, the answer to this question is positive, but it would lie beyond the scope of this paper to examine the $\mathsf{ST}$-hierarchy, and a presentation of these results and of their proof is postponed to another occasion.



Several questions remain open. The most pressing one for us concerns the possibility, discussed in \cite{dare2023three}, of defining the logic $\mathsf{ST}$ using other schemes than the Strong-Kleene scheme. Over those more extended schemes, does the characterization of $\mathsf{ST}$ in terms of relational product remain adequate? And can it be strengthened using the interpolation property in all cases? We have obtained some preliminary results on this matter, but more work needs to be done before we can answer these questions in full generality.\bigskip







\subsection*{Acknowledgments}

{\small We thank Miguel Alvarez, Eduardo Barrio, Pablo Cobreros, Bruno Da R\'e, Francesco Paoli, David Ripley, Robert van Rooij, Mariela Rubin, Damian Szmuc, Allard Tamminga, for helpful exchanges related to the topic of this paper. Special thanks to P. Cobreros, D. Ripley, and R. van Rooij, for early stimulating discussions with PE held in particular at ESSLLI 2011 in Ljubljana. 
We also acknowledge the members of the Buenos Aires Logic Group, and several audiences, in particular at the PLEXUS inaugural conference held in Lisbon in May 2023. This collaboration received support from the program ECOS-SUD (``Logical consequence and many-valued models'', no. A22H01), a bilateral exchange program between Argentina and France, and from PLEXUS (``Philosophical, Logical and Experimental Perspectives on Substructural Logics'', grant agreement no. 101086295), a Marie Sk\l odowska-Curie action funded by the EU under the Horizon Europe Research and Innovation Program. We are grateful to the IIF-SADAF-CONICET and to Monash University for hosting each of us, respectively, during the writing of this paper. We also thank the program ANR-17-EURE0017 (FrontCog) for research conducted at the Department of Cognitive Studies of ENS-PSL.}

\newpage
\printbibliography

\end{document}